\documentclass[aos]{imsart}

\RequirePackage{amsthm,amsmath,amsfonts,amssymb}
\RequirePackage[numbers]{natbib}
\usepackage{mathrsfs}
\usepackage{hyperref}
\usepackage{enumitem}
\usepackage[cal=boondoxo]{mathalpha}
\usepackage{xcolor}
\usepackage[final]{pdfpages}

\startlocaldefs
\theoremstyle{plain}  
\newtheorem{theorem}{Theorem}
\newtheorem{lemma}[theorem]{Lemma}
\newtheorem{corollary}[theorem]{Corollary}
\newtheorem{proposition}[theorem]{Proposition}

\theoremstyle{remark}

\newtheorem{remark}[theorem]{Remark}

\newcommand{\df}{\mathrm{d}}
\newcommand{\X}{\mathsf{X}}
\newcommand{\Y}{\mathsf{Y}}

\newcommand{\B}{\mathcal{B}}

\newcommand{\F}{\mathcal{F}}

\newcommand{\ind}{\mathbf{1}}
\newcommand\emm{\mathcal{m}}

\makeatletter
\def\namedlabel#1#2{\begingroup
	#2%
	\def\@currentlabel{#2}%
	\phantomsection\label{#1}\endgroup
}
\makeatother

\endlocaldefs

\begin{document}
	
	\begin{frontmatter}
		\title{Spectral gap bounds for reversible hybrid Gibbs chains}
		\runtitle{Hybrid Gibbs chains}
		
		\begin{aug}
			\author[A]{\fnms{Qian}~\snm{Qin}\ead[label=e1]{qqin@umn.edu}},
			\author[B]{\fnms{Nianqiao}~\snm{Ju}\ead[label=e2]{nianqiao@purdue.edu }}
			\and
			\author[C]{\fnms{Guanyang}~\snm{Wang}\ead[label=e3]{gw295@stat.rutgers.edu}}
			\address[A]{School of Statistics, University of Minnesota\printead[presep={,\ }]{e1}}
			
			\address[B]{Department of Statistics, Purdue University\printead[presep={,\ }]{e2}}
			
			\address[C]{Department of Statistics, Rutgers University\printead[presep={,\ }]{e3}}
		\end{aug}
		
		\begin{abstract}
			Hybrid Gibbs samplers represent a prominent class of approximated Gibbs algorithms that utilize Markov chains to approximate conditional distributions, with the Metropolis-within-Gibbs algorithm standing out as a well-known example. 
			Despite their widespread use in both statistical and non-statistical applications, little is known about their convergence properties.
			This article introduces novel methods for establishing bounds on the convergence rates of certain reversible hybrid Gibbs samplers. 
			In particular, we examine the convergence characteristics of hybrid random-scan Gibbs algorithms. 
			Our analysis reveals that the absolute spectral gap of a hybrid Gibbs chain can be bounded based on the absolute spectral gap of the exact Gibbs chain and the absolute spectral gaps of the Markov chains employed for conditional distribution approximations. 
			We also provide a convergence bound of similar flavors for hybrid data augmentation algorithms, extending existing works on the topic.
			The general bounds are applied to three examples: a random-scan Metropolis-within-Gibbs sampler, random-scan Gibbs samplers with block updates, and a hybrid slice sampler.
		\end{abstract}
		
		\begin{keyword}[class=MSC]
			\kwd[Primary ]{60J05}
		\end{keyword}
		
		\begin{keyword}
			\kwd{convergence rate}
			\kwd{data augmentation}
			\kwd{Metropolis-within-Gibbs}
			\kwd{MCMC}
			\kwd{slice sampler}
		\end{keyword}
		
	\end{frontmatter}
	

	\section{Introduction} \label{sec:intro}
	
	The Gibbs sampler \citep{geman1984stochastic,casella1992explaining} is an immensely popular Markov chain Monte Carlo (MCMC) algorithm.
	It can be used to sample from the joint distribution of multiple variables by iteratively updating each variable based on its conditional distribution while keeping the other variables fixed.
	In practice, the conditional distributions themselves may be difficult to sample from.
	Algorithms like the Metropolis-within-Gibbs algorithm \citep{metropolis1953equation,chib1995understanding,tierney1998note} circumvent this problem by approximating an intractable conditional distribution with a Markovian step that leaves the conditional distribution invariant.
	We refer to these algorithms as "hybrid" Gibbs algorithms, as opposed to "exact" Gibbs algorithms that make perfect draws from conditional distributions.
	Since their inception, hybrid Gibbs samplers have been exploited to solve many complex problems from various domains, including pharmacokinetics~\citep{gilks1995adaptive}, biomedical informatics~\citep{wang2023bayesian}, econometrics~\citep{Schamberger2017bayesian}, deep latent variable models~\citep{mattei2018leveraging}, and privacy-protected data analysis~\citep{ju2022data}.

	Despite the pervasive popularity and extensive application of these hybrid algorithms, significant gaps persist in our understanding of their convergence properties. 
	Building on the foundational works of \citet{roberts1997geometric} and \citet{andrieu2018uniform}, this article aims to bridge the gap by establishing a connection between the convergence rates of exact and hybrid Gibbs chains in some important scenarios.
	Intuitively, if the Markovian steps in a hybrid Gibbs sampler approximate the intractable conditional distributions well, then the convergence rate of the hybrid Gibbs sampler should be similar to that of the corresponding exact Gibbs sampler.
	Our analysis confirms this intuition and makes this notion precise with some quantitative convergence rate inequalities.
	In particular, we provide bounds on the absolute spectral gap of the hybrid algorithm in terms of that of the exact algorithm and the quality of the approximation.
	Bounds involving Dirichlet form and asymptotic variance are also provided.
	
	We focus on the hybrid random-scan Gibbs algorithm, but also touch on hybrid data augmentation (two-component Gibbs) algorithms.
	The Markovian approximation of a conditional distribution is assumed to be reversible.
	This is the case for Metropolis-within-Gibbs samplers.
	
	The efficiency of a reversible MCMC sampler is closely related to its spectral gap.
	For a Markov chain associated with a Markov transition kernel (Mtk)~$K$ that is reversible with respect to a distribution~$\omega$, its absolute spectral gap is defined to be $1 - \|K\|_{\omega}$, where $\|K\|_{\omega} \in [0,1]$ is the $L^2$ norm of~$K$ when~$K$ is viewed as a Markov operator.
	(See Section~\ref{sec:preliminaries} for the exact definitions.)
	The $L^2$ distance between the distribution of the chain's $t$th iteration 
	and its target~$\omega$ decreases at a rate of $\|K\|_{\omega}^t$ as~$t$ tends to infinity.
	Indeed, the norm $\|K\|_{\omega}$ can be regarded as the convergence rate of the chain, with a smaller rate indicating faster convergence \citep{roberts1997geometric}.
	The number of iterations it takes for the chain to reach within some $L^2$ distance of its stationary distribution is roughly proportional to $-1/\log \|K\|_{\omega}$, which is approximately $1/(1-\|K\|_{\omega})$ when $\|K\|_{\omega}$ is not too close to~0.
	Our results give quantitative upper and lower bounds on $(1-\|K_2\|_{\omega})/(1-\|K_1\|_{\omega})$ when $K_1$ and $K_2$ are the Mtks of a Gibbs sampler and its hybrid version.

	In Section \ref{sec:examples}, we apply the spectral gap bounds, derived under general conditions, to three examples, illustrating their relevance and potential for various applications. Together, these examples effectively showcase how our framework enriches and diversifies the analysis of exact and hybrid Gibbs samplers across different scenarios.
	\begin{itemize}[leftmargin = *]
		\item 	Firstly, the methods herein can be efficiently integrated with existing techniques to establish new and explicit convergence bounds. 
		In Section  \ref{ssec:met}, we study random-scan Gibbs samplers with conditional distributions approximated using a random-walk Metropolis algorithm.
		We derive an explicit relationship between the spectral gaps of the exact and hybrid samplers when the conditional distributions are log-concave using recent studies on log-concave sampling \citep{andrieu2022explicit}. 
		This relation is applied to a Bayesian regression model with spike and slab priors.
		
		\item Secondly, our theory can be used to generalize existing frameworks for analyzing Gibbs algorithms. 
		In Section \ref{ssec:hie},  we apply our method to analyze	random-scan Gibbs samplers that update multiple coordinates at a time. In particular, we provide a quantitative relationship between two samplers that update, respectively,~$\ell$ and~$\ell'$ coordinates at a time, in terms of their spectral gaps. This extends the main result (Theorem 2) in \cite{qin2023spectral},	 which itself is a generalization of a breakthrough technique called "spectral independence"  recently developed in theoretical computer science for analyzing the convergence properties of Gibbs algorithms \citep{anari2021spectral, chen2021optimal, blanca2022mixing, feng2022rapid}.
		
		\item Lastly, our technique may improve existing studies of some well-known problems. 	
		In  Section \ref{ssec:sli}, we revisit a hybrid slice sampler previously studied by \citet{latuszynski2014convergence}.
		By viewing it as a hybrid Gibbs sampler with two components, we derive a convergence rate bound that is similar but slightly better than the one obtained in \cite{latuszynski2014convergence}.
	\end{itemize}

	We now review some relevant existing works and explain how our contribution relates to them.
	On finite state spaces, \citet{liu1996peskun} constructed a Metropolized random-scan Gibbs sampler and showed that it is asymptotically more efficient than the corresponding exact Gibbs sampler using Peskun's ordering \citep{peskun1973optimum}. 
	Analysis on general state spaces is more challenging.
	The convergence properties of hybrid Gibbs samplers were studied qualitatively by~\citet{roberts1997geometric}, \citet{roberts1998two}, \citet{fort2003geometric}, and \citet{qin2021convergence}. 
	Quantitative studies on the subject remain scarce.
	When the approximating Markov chains are uniformly ergodic with convergence rates bounded away from one, some quantitative convergence bounds are given in \cite{roberts1998two}, \cite{neath2009variable}, and \cite{jones2014convergence}.
	For hybrid random-scan Gibbs algorithms, we develop a spectral gap bound under the weaker assumption that the approximating chains are geometrically ergodic with convergence rates bounded away from unity.
	This is our main contribution.
	
	We also develop a new spectral gap bound for hybrid data augmentation algorithms.
	{ For these algorithms, \citet{andrieu2018uniform} gave results leading to a spectral gap bound that is similar to our bound for hybrid random-scan Gibbs algorithms in form.
		Building on their work, we provide an alternative form of bound for hybrid data augmentation algorithms.}
	
	The proofs of our results utilize the idea of Markov decomposition \citep{caracciolo1992two,madras2002markov}.
	Some of the linear algebraic techniques that we employ are adopted from \cite{andrieu2018uniform} and \cite{latuszynski2014convergence}.
	Following the initial release of our work, \citet{ascolani2024scalability} derived bounds for the $s$-conductance of a class of hybrid random-scan Gibbs algorithms, exhibiting some similarities to our spectral gap bounds.

	The rest of this article is organized as follows.
	An overview of our results is given in Section~\ref{sec:overview}.
	Section~\ref{sec:preliminaries} gives a brief survey on the $L^2$ theory of Markov chains, and contains the definitions of concepts like Markov operator and spectral gap.
	The main result for random-scan Gibbs samplers are formally stated in Section~\ref{sec:main}.
	Section~\ref{sec:da} contains results involving data augmentation samplers.
	In Section~\ref{sec:examples} we apply results in Sections \ref{sec:main} and \ref{sec:da} to three examples.
	Section~\ref{sec:discussion} contains a brief discussion on avenues for future research.

	The Appendix includes technical results and proofs.
	In a supplementary file, we derive a spectral gap bound for a type of hybrid proximal sampler using techniques from \cite{andrieu2018uniform}, rather than the methods from Sections \ref{sec:main} and \ref{sec:da}, which may be of interest to practitioners. 
	In the same file, we derive an alternative spectral gap bound for hybrid random-scan Gibbs algorithms using arguments from \cite{roberts1997geometric}, and compare it to the bound derived in Section \ref{sec:main}.
	While this alternative bound performs less favorably in high-dimensional settings, its derivation based on \cite{roberts1997geometric} may still offer valuable insights.

	\section{Overview} \label{sec:overview}

	In this section, we describe two hybrid Gibbs samplers of interest, and give an overview of our results.

	\subsection{Gibbs and hybrid Gibbs samplers} \label{ssec:samplers}

	{
		Suppose that $(\X,\B)$ is the product of~$n$ measurable spaces, say $((\X_i,\B_i))_{i=1}^n$.
		Assume that each $\X_i$ is Polish, and $\B_i$ is its Borel algebra.
		Let $\Pi$ be a probability measure on $(\X,\B)$.
		For $i \in \{1,\dots,n\}$ and $(x_1, \dots, x_n) \in \X$, let $x_{-i} = (x_1, \dots, x_{i-1}, x_{i+1}, \dots, x_n)$, and let $\X_{-i} = \{x_{-i}: \, x \in \X\}$.
		If $X \sim \Pi$, denote by $M_i$ the marginal distribution of $X_i$, by $M_{-i}(\cdot)$ the marginal distribution of $X_{-i}$, and by $\Pi_{i, y}(\cdot)$ the conditional distribution of $X_i$ given $X_{-i} = y \in \X_{-i}$.
		Then we have the disintegration
		\[
		\Pi(\df x) = M_{-i}(\df x_{-i}) \, \Pi_{i, x_{-i}}(\df x_i).
		\]
	}

	A random-scan Gibbs chain targeting~$\Pi$ is a Markov chain with transition kernel $T: \X \times \B \to [0,1]$ given by
	\begin{equation} \label{eq:T}
		T(x, \df x') = \sum_{i=1}^n p_i \, \Pi_{i, x_{-i}}(\df x_i' ) \, \delta_{x_{-i}}(\df x_{-i}'),
	\end{equation}
	where $p_1, \dots, p_n$ are non-negative constants such that $\sum_{i=1}^n p_i = 1$, and $\delta_y(\cdot)$ denotes the point mass (Dirac measure) at $y \in \X_{-i}$.
	To simulate the chain, given the current state~$x$, one randomly draws an index~$i$ from $\{1,\dots,n\}$ with probability $p_i$, and updates $x_i$ through $\Pi_{i,x_{-i}}(\cdot)$.
	It is well-known that $T$ is reversible with respect to~$\Pi$.

	In practice, it may be difficult to sample exactly from $\Pi_{i,x_{-i}}(\cdot)$, preventing one from simulating the above Gibbs chain.
	Suppose that, instead, for $\Pi$-a.e. $x \in \X$, it is possible to sample from $Q_{i,x_{-i}}(x_i, \cdot)$, where, for $y \in \X_{-i}$, $Q_{i, y}: \X_i \times \B_i \to [0,1]$ is an Mtk that is reversible with respect to $\Pi_{i,y}$.
	One can then define a hybrid random-scan Gibbs chain with the Mtk
	\begin{equation} \label{eq:That}
		\hat{T}(x, \df x')  = \sum_{i=1}^n p_i \, Q_{i, x_{-i}}(x_i, \df x_i') \, \delta_{x_{-i}}(\df x_{-i}')
	\end{equation}
	Typically, $Q_{i,y}$ is associated with a Metropolis-Hastings sampler targeting $\Pi_{i,y}$, and the resultant hybrid Gibbs sampler is called a Metropolis-within-Gibbs sampler.
	To ensure that $\hat{T}$ is well-defined, we impose the following regularity condition:
	\begin{itemize}
		\item [\namedlabel{A1}{(A)}]
		For $i = 1, \dots,n$, and $A \in \B_i$, 
		the function $x \mapsto Q_{i,x_{-i}}(x_i, A)$ is measurable.
	\end{itemize}
	It is not difficult to show that $\hat{T}$ is reversible with respect to~$\Pi$.

	Another type of Gibbs chain that we shall investigate is a marginal deterministic-scan Gibbs chain with two components, which defines a data augmentation algorithm \citep{tanner1987calculation,van2001art,hobe:2011}.
	To be specific, suppose that $n = 2$, and let $S: \X_1 \times \B_1 \to [0,1]$ be an Mtk of the form
	\[
	S(y, \df y') = \int_{\X_2} \Pi_{2,y}(\df z) \, \Pi_{1, z}(\df y').
	\]
	To simulate this chain, given the current state $y$, one first draws $z$ from the conditional distribution $\Pi_{2,y}(\cdot)$, then draws the next state $y'$ from the other conditional distribution $\Pi_{1, z}(\cdot)$.
	It is well-known that $S$ is reversible with respect to $M_1$.

	Suppose that one can efficiently sample from $\Pi_{2,y}( \cdot)$ for $y \in \X_1$, but not from $\Pi_{1, z}( \cdot)$ for at least some $z \in \X_2$.
	However, assume that, given $z \in \X_2$, one can sample from $Q_{1,z}(y, \cdot)$ for $y \in \X_1$, where $Q_{1,z}$ is an Mtk that is reversible with respect to $\Pi_{1,z}$.
	Then one can simulate the hybrid data augmentation chain with Mtk
	\[
	\hat{S}(y, \df y') = \int_{\X_2} \Pi_{2,y}(\df z) \, Q_{1, z}(y, \df y').
	\]
	In this algorithm,  given the current state $y$, one first draws~$z$ from the conditional distribution $\Pi_{2,y}( \cdot)$, then updates~$y$ to~$y'$ by simulating one step of the chain associated with the Mtk $Q_{1,z}$.
	It can be shown that $\hat{S}$ is reversible with respect to $M_1$ \citep[see, e.g.,][Section 2.4]{jones2014convergence}.

	\subsection{Summary of results} \label{ssec:loose}
	
	We now give an overview of our results.
	The exact statements are given in Sections~\ref{sec:main} and~\ref{sec:da}, after a brief survey on the $L^2$ theory of Markov chains in Section~\ref{sec:preliminaries}.
	For now, recall that the $L^2$ norm of a Markov operator is in $[0,1]$.
	When the associated chain is reversible, the norm can be viewed as its convergence rate, with a smaller rate indicating faster convergence to the stationary distribution.
	It is also common to quantify the convergence speed of chains through the absolute spectral gap, which is one minus the operator norm (see Section~\ref{sec:preliminaries}).
	The norms of $T$, $\hat{T}$, $S$, $\hat{S}$, and $Q_{i,y}$ are denoted by $\|T\|_{\Pi}$, $\|\hat{T}\|_{\Pi}$, $\|S\|_{M_1}$, $\|\hat{S}\|_{M_1}$, and $\|Q_{i,y}\|_{\Pi_{i,y}}$, respectively.

	For the random-scan Gibbs sampler, our spectral gap bound loosely goes like this:
	\begin{equation} \label{ine:Tbound-loose}
		\left(1 - \sup_{i,y} \|Q_{i,y}\|_{\Pi_{i,y}} \right) (1 - \|T\|_{\Pi}) \leq 1 - \|\hat{T}\|_{\Pi} \leq \left(1 + \sup_{i,y} \|Q_{i,y}\|_{\Pi_{i,y}} \right) (1 - \|T\|_{\Pi}),
	\end{equation}
	where the superemum is taken over $\bigcup_{i=1}^n \{i\} \times \X_{-i}$.
	If the Markov operator $Q_{i,y}$ is positive semi-definite for each $i$ and $y$, then the right-hand-side can be tightened to $1 - \|T\|_{\Pi}$.
	
	The bounds in \eqref{ine:Tbound-loose} imply that if the convergence rates of all chains associated with $Q_{i,y}$ are uniformly bounded away from~1 when~$i$ and $y$ vary, then the ratio of the spectral gaps of~$\hat{T}$ and~$T$ is in $(0,2)$.
	In this case, exact Gibbs and hybrid Gibbs have similar/comparable convergence rates.
	
	The lower bound in \eqref{ine:Tbound-loose} has the form
	\begin{equation} \label{ine:general}
		\left( \inf_i \text{Spectral gap of chain } C_i \right) \times \text{Spectral gap of chain } A \leq \text{Spectral gap of chain } B.
	\end{equation}
	Interestingly, bounds of this form have appeared in various other settings that are quite distinct from the one at hand \citep{andrieu2018uniform,madras1999importance,madras2002markov,carlen2003determination,qin2023geometric}.
	The similarity might be attributed to the presence of the idea of Markov chain decomposition \cite{caracciolo1992two} in their derivations.
	
	Previously, the only quantitative convergence bound for general hybrid random-scan Gibbs samplers that we know of is from \cite{roberts1998two}.
	\citet{roberts1998two} established a relation that bears some similarity to the first inequality in \eqref{ine:Tbound-loose}.
	However, their proof relied on the strong assumptions that the chain associated with $Q_{i,y}$ is uniformly ergodic for all~$i$ and $y$, and the chain associated with~$T$ is uniformly ergodic as well.

	{
		For hybrid data augmentation chains, the following convergence bound, which is also of the form \eqref{ine:general}, was derived by \citet{andrieu2018uniform}.
		\begin{equation} \label{ine:Sbound-loose}
			\left(1 - \sup_{z \in \X_2} \|Q_{1,z}\|_{\Pi_{1,z}} \right) (1 - \|S\|_{M_1}) \leq 1 - \|\hat{S}\|_{M_1} \leq \left(1 + \sup_{z \in \X_2} \|Q_{1,z}\|_{\Pi_{1,z}} \right) (1 - \|S\|_{M_1}).
		\end{equation}
		If the Markov operator $Q_{1,z}$ is positive semi-definite for each $z$, then the right-hand-side can be tightened to $1 - \|S\|_{M_1}$.
		In fact, \citet{andrieu2018uniform} derived a more general form of the bound.
		The general form involves an intractable quantity, and bounding the intractable quantity in an obvious manner yields \eqref{ine:Sbound-loose}.
	}

	For the lower bound on $1 - \|\hat{S}\|_{M_1}$ in \eqref{ine:Sbound-loose} to be non-trivial (strictly greater than~0), it is necessary that $\sup_{z} \|Q_{1,z}\|_{\Pi_{1,z}} < 1$.
	Our contribution regarding hybrid data augmentation algorithms relaxes this requirement.
	By combining arguments from \cite{latuszynski2014convergence} and \cite{andrieu2018uniform}, we derive the following bound, which holds any positive even integer~$t$:
	\begin{equation} \label{ine:Sbound-loose-2}
		t (1 - \|\hat{S}\|_{M_1}) \geq 1 - \|\hat{S}\|_{M_1}^t \geq 1 - \|S\|_{M_1} - \sup_{y \in \X_1} \int_{\X_2} \|Q_{1,z}\|_{\Pi_{1,z}}^t \Pi_{2,y}( \df z).
	\end{equation}
	If $Q_{1,z}$ is positive semi-definite for each $z$, then \eqref{ine:Sbound-loose-2} holds for odd values of $t$ as well.
	When $\|S\|_{M_1} < 1$, the second inequality is nontrivial for large values of $t$ when
	\[
	\lim_{t \to \infty} \sup_{y \in \X_1} \int_{\X_2} \|Q_{1,z}\|_{\Pi_{1,z}}^t \Pi_{2,y}( \df z) = 0.
	\]
	This requirement is evidently weaker than $\sup_{z \in \X_2} \|Q_{1,z}\|_{\Pi_{1,z}} < 1$.


	\begin{remark}
		For the lower bound in \eqref{ine:Tbound-loose} to be nontrivial, it is necessary that $\sup_{i,y} \|Q_{i,y}\|_{\Pi_{i,y}} < 1$.
		We were not able to obtain an analogue of~\eqref{ine:Sbound-loose-2} for the random-scan Gibbs chain that is capable of alleviating this restriction.
	\end{remark}

	\section{Preliminaries} \label{sec:preliminaries}
	
	This section reviews the $L^2$ theory for Markov chains.
	
	Let $(\Y, \F, \omega)$ be a probability space.
	The linear space $L^2(\omega)$ consists of functions  $f: \Y \to \mathbb{R}$ such that $\int_{\Y} f(y)^2 \, \omega(\df y) < \infty$.
	Two functions in $L^2(\omega)$ are considered identical if they are equal almost everywhere.
	For $f, g \in L^2(\omega)$ and $c \in \mathbb{R}$, $(f+g)(y) = f(y) + g(y)$, and $(cf)(y) = cf(y)$.
	The inner product of~$f$ and~$g$ is $\langle f, g \rangle_{\omega} = \int_{\Y} f(y) g(y) \, \omega(\df y)$.
	The norm of~$f$ is $\|f\|_{\omega} = \sqrt{\langle f, f \rangle_{\omega}}$.
	Let $L_0^2(\omega)$ be the subspace of $L^2(\omega)$ consisting of functions~$f$ such that $\omega f := \int_{\Y} f(y) \, \omega(\df y) = 0$.
	{ It is well known that $(L^2(\omega), \langle \cdot, \cdot, \rangle_{\omega} )$ and $(L_0^2(\omega), \langle \cdot, \cdot, \rangle_{\omega} )$ are real Hilbert spaces \citep[see, e.g.,][Theorem 13.15]{bruckner2008}.}
	
	A probability measure $\mu: \F \to [0,1]$ is said to be in $L_*^2(\omega)$ if $\df \mu/ \df \omega$ exists and is in $L^2(\omega)$.
	The $L^2$ distance between $\mu, \nu \in L_*^2(\omega)$ is
	\[
	\|\mu - \nu\|_{\omega} := \sup_{f \in L_0^2(\omega)} (\mu f - \nu f) = \left\| \frac{\df \mu}{\df \omega} - \frac{\df \nu}{\df \omega} \right\|_{\omega},
	\]
	where $\mu f$ and $\nu f$ are defined analogously to $\omega f$.
	
	Let $K: \Y \times \F \to [0,1]$ be an Mtk, i.e., $K(x,\cdot)$ is a probability measure for $x \in \X$ and $K(\cdot, A)$ is a measurable function for $A \in \F$.
	Assume that $\omega(\cdot)$ is stationary for $K(\cdot,\cdot)$, i.e.,
	\[
	\omega K(\cdot) := \int_{\Y} \omega(\df y) \, K(y, \cdot) = \omega(\cdot).
	\]
	For any positive integer~$t$, let $K^t$ be the $t$-step Mtk, so that $K^1 = K$ and $K^{t+1}(y, \cdot) = \int_{\Y} K(y,\df z) K^t(z, \cdot)$.
	It is well-known that $K$ can be seen as a linear operator on $L_0^2(\omega)$:
	for $f \in L_0^2(\omega)$,
	\[
	K f(\cdot) = \int_{\Y} K(\cdot, \df y) f(y).
	\]
	We refer to~$K$ as a Markov operator.
	The operator corresponding to the $t$-step Mtk $K^t(\cdot,\cdot)$ is precisely the operator $K^t = \prod_{i=1}^t K$.
	By Jensen's inequality, the norm of~$K$ satisfies
	\[
	\|K\|_{\omega} := \sup_{f \in L_0^2(\omega) \setminus \{0\}} \frac{\|Kf\|_{\omega}}{\|f\|_{\omega}} \leq 1.
	\]
	The operator~$K$ is self-adjoint, i.e., $\langle K f, g \rangle_{\omega} = \langle f, K g \rangle_{\omega}$ for $f, g \in L_0^2(\omega)$,
	if and only if the corresponding chain is reversible with respect to~$\omega$.
	We say~$K$ is positive semi-definite if it is self-adjoint, and $\langle f, Kf \rangle_{\omega} \geq 0$ for $f \in L_0^2(\omega)$.
	{ One can show that~$T$ and~$S$ from Section~\ref{sec:overview} are both positive semi-definite linear operators \citep[see, e.g.,][]{rudolf2013positivity, liu1995covariance}.}
	
	Assume for the rest of this section that $K$ is self-adjoint.
	For $f \in L^2(\omega)$,
	\[
	\mathcal{E}_{K}(f)  = \|f - \omega f\|_{\omega}^2 - \langle f - \omega f, K (f - \omega f) \rangle_{\omega} = \frac{1}{2} \int_{\Y^2} \omega(\df y) K(y, \df y') [f(y) - f(y')]^2
	\] 
	is the Dirichlet form of $f$ with respect to $K$.
	By the Cauchy-Schwarz inequality and the fact that $\|K\|_{\omega} \leq 1$, for $f \in L_0^2(\omega) \setminus \{0\}$, $\mathcal{E}_K(f) / \|f\|_{\omega}^2 \in [0,2]$, and, if $K$ is positive semi-definite, $\mathcal{E}_K(f) / \|f\|_{\omega}^2 \in [0,1]$.
	{
		The absolute spectral gap of~$K$ is defined to be one minus $K$'s spectral radius.
		Since $K$ is self-adjoint, this equates to $1 - \|K\|_{\omega}$ \citep[see, e.g.,][\S30, Corollary 8.1]{helmberg2014introduction}.
		By well-known properties of self-adjoint operators \citep[see, e.g.,][\S14, Corollary 5.1]{helmberg2014introduction},  
		\begin{equation} \label{eq:norm-def}
			\begin{aligned}
				\|K\|_{\omega} &= \sup_{f \in L_0^2(\omega) \setminus \{0\}} \left| \frac{\langle f, Kf \rangle_{\omega}}{\|f\|_{\omega}^2} \right|, \\
				1 - \|K\|_{\omega} &= \min \left\{ 2 - \sup_{f \in L_0^2(\omega) \setminus \{0\}} \frac{\mathcal{E}_K(f)}{\|f\|_{\omega}^2}, \inf_{f \in L_0^2(\omega) \setminus \{0\}} \frac{\mathcal{E}_{K}(f)}{\|f\|_{\omega}^2} \right\}.
			\end{aligned}
		\end{equation}
	}

	One can find $C: L_*^2(\omega) \to [0,\infty)$ such that
	\[
	\|\mu K^t - \omega\|_{\omega} \leq C(\mu) \rho^t
	\]
	for every $\mu \in L_*^2(\omega)$ and $t \geq 1$ if and only if $\rho \geq \|K\|_{\omega}$,
	where $\mu K^t (\cdot) = \int_{\Y} \mu(\df y) \, K^t(y, \cdot)$ \citep{roberts1997geometric}.
	In this sense, $\|K\|_{\omega}$ defines the convergence rate of the corresponding reversible chain.
	
	{
		When a chain $(X(t))_{t=0}^{\infty}$ evolves according to $K$ and $X(0) \sim \omega$, the asymptotic variance of a function $f \in L^2(\omega)$ is $\mbox{var}_K(f) := \lim_{t \to \infty} t^{-1/2} \mbox{var}[ \sum_{t=1}^t f(X(t)) ]$.
		Note that $\mbox{var}_K(f)  = \mbox{var}_K(f - \omega f)$.
		When $\|K\|_{\omega} < 1$, i.e., the Markov chain associated with~$K$ is geometrically convergent in the $L^2$ distance \citep[Theorem 2.1]{roberts1997geometric}, for $f \in L_0^2(\omega)$, 
		\[
		\mbox{var}_K(f) = 2 \langle f, (I-K)^{-1} f \rangle_{\omega} - \|f\|_{\omega}^2,
		\]
		where $I$ is the identity operator on $L_0^2(\omega)$ \citep[see, e.g.,][Proposition A.1]{caracciolo1990nonlocal}.
		Note that, in the above equation, $\|f\|_{\omega}^2$ is the variance of $f(X)$ if $X \sim \omega$.
	}

	\section{Random-scan Gibbs chain} \label{sec:main}
	
	In this section, we derive our main result concerning the relationship between exact and hybrid random-scan Gibbs samplers.
	
	Recall that $T$ as defined in \eqref{eq:T} is the Mtk of the random-scan Gibbs chain, which is reversible with respect to~$\Pi$.
	Thus,~$T$ can be regarded as a self-adjoint operator on $L_0^2(\Pi)$.
	In fact, it is well-known that $T$ is a mixture of orthogonal projections \cite{greenwood1998information} {(an operator $P$ on $L_0^2(\Pi)$ is an orthogonal projection if $P$ is self-adjoint and idempotent in the sense that $P^2 = P$)}, and thus positive semi-definite.
	The convergence rate of the Gibbs chain is given by $\|T\|_{\Pi}$.
	For $y \in \X_{-i}$, the Mtk $Q_{i,y}$ is reversible with respect to $\Pi_{i,y}$.
	The convergence rate of the corresponding chain is given by $\|Q_{i,y}\|_{\Pi_{i, y}}$.
	The hybrid Gibbs chain of interest is associated with the Mtk $\hat{T}$ defined in \eqref{eq:That}.
	Recall that this chain is reversible with respect to~$\Pi$, so its convergence rate is given by $\|\hat{T}\|_{\Pi}$.
	
	Our goal is to bound $\|\hat{T}\|_{\Pi}$ in terms of $\|T\|_{\Pi}$ and an upper bound on $\|Q_{i,y}\|_{\Pi_{i, y}}$.
	We will rely on the following key technical result.

	\begin{theorem} \label{thm:random}
		Suppose that there exist constants $c_1, c_2 \in [0,2]$ such that
		\begin{equation} \label{ine:Qiyi-bound}
			c_1 \leq \frac{\mathcal{E}_{Q_{i,y}}(g)}{\|g\|_{\Pi_{i,y}}^2} \leq c_2
		\end{equation}
		for $i=1,\dots,n$, $M_{-i}$-a.e. $y \in \X_{-i}$, and $g \in L_0^2(\Pi_{i,y}) \setminus \{0\}$.
		Then, for $f \in L_0^2(\Pi)$,
		\begin{equation} \label{ine:That-Dir}
			c_1 \, \mathcal{E}_{T}(f) \leq \mathcal{E}_{\hat{T}}(f) \leq c_2 \, \mathcal{E}_{T}(f).
		\end{equation}
	\end{theorem}
	
	{
		
		The proof of Theorem \ref{thm:random} builds on the idea of decomposing Dirichlet forms, which can be traced back to \cite{caracciolo1992two} and \cite{madras1999importance}.
		Similar techniques were employed in \cite{andrieu2018uniform} and \cite{qin2023spectral} to study hybrid data augmentation algorithms and block Gibbs samplers.
		
		\begin{proof}

			Let $f \in L_0^2(\Pi)$.
			For $i \in \{1,\dots,n\}$ and $y \in \X_{-i}$, let $f_{i,y}: \X_i \to \mathbb{R}$ be such that $f_{i,y}(x_i) = f(x)$ if $x_{-i} = y$.
			Then $f(x) = f_{i,x_{-i}}(x_i)$ for $x \in \X$.
			Since 
			\[
			\|f\|_{\Pi}^2 = \int_{\X_{-i}} \int_{\X_i} f_{i, x_{-i}}(x_i)^2 \, \Pi_{i,x_{-i}}(\df x_i) \, M_{-i}(\df x_{-i}) < \infty,
			\]
			for $M_{-i}$-a.e. $y \in \X_{-i}$,
			the function $f_{i,y}$ can be regarded as an element of $L^2(\Pi_{i,y})$, and $f_{i,y} - \Pi_{i,y} f_{i,y}$ can be seen as an element of $L_0^2(\Pi_{i,y})$.
			
			By the definition of Dirichlet forms, it is easy to derive
			\begin{equation} \label{eq:Thatff-1}
				\begin{aligned}
					\mathcal{E}_{\hat{T}}(f) =& \frac{1}{2} \sum_{i=1}^n p_i \int_{\X_{-i}} M_{-i}(\df x_{-i}) \int_{\X_i^2} \Pi_{i,x_{-i}}(\df x_i) \, Q_{i,x_{-i}}(x_i, \df x_i') \,  [f_{i,x_{-i}}(x_i) - f_{i,x_{-i}}(x_i')]^2 \\
					=& \sum_{i=1}^n p_i \int_{\X_{-i}} M_{-i}(\df x_{-i}) \, \mathcal{E}_{Q_{i,x_{-i}}}(f_{i,x_{-i}}) \\
					=& \sum_{i=1}^n p_i \int_{\X_{-i}} M_{-i}(\df x_{-i}) \, \mathcal{E}_{Q_{i,x_{-i}}}( f_{i,x_{-i}} - \Pi_{i,x_{-i}} f_{i,x_{-i}} ).
				\end{aligned}
			\end{equation}
			On the other hand, 
			\begin{equation} \nonumber
				\begin{aligned}
					&\int_{\X_{-i}} M_{-i}(\df x_{-i}) \, \| f_{i,x_{-i}} - \Pi_{i,x_{-i}} f_{i,x_{-i}} \|_{\Pi_{i,x_{-i}}}^2  \\
					=& \int_{\X_{-i}} M_{-i}(\df x_{-i}) \left[ \| f_{i,x_{-i}}  \|_{\Pi_{i,x_{-i}}}^2 - (\Pi_{i,x_{-i}} f_{i,x_{-i}})^2 \right]  \\
					=& \int_{\X_{-i}} M_{-i}(\df x_{-i}) \int_{\X_i} \Pi_{i,x_{-i}}(\df x_i) \, f(x)^2  - \\
					& \int_{\X_{-i}} M_{-i}(\df x_{-i}) \int_{\X_i} \Pi_{i,x_{-i}}(\df x_i) \, f_{i,x_{-i}}(x_i) \, \int_{\X_i} \Pi_{i,x_{-i}}(\df x_i') \, f_{i,x_{-i}}(x_i') \\
					=& \|f\|_{\Pi}^2 - \int_{\X} \Pi(\df x) \, f(x) \int_{\X} \Pi_{i,x_{-i}}(\df x_i') \, \delta_{x_{-i}}(\df x_{-i}') \, f(x').
				\end{aligned}
			\end{equation}
			This implies that
			\begin{equation} \label{eq:Dirichlet-T}
				\mathcal{E}_{T}(f) = \sum_{i=1}^n p_i \int_{\X_{-i}} M_{-i}(\df x_{-i})  \, \| f_{i,x_{-i}} - \Pi_{i,x_{-i}} f_{i,x_{-i}} \|_{\Pi_{i,x_{-i}}}^2.
			\end{equation}
			Comparing \eqref{eq:Thatff-1} and \eqref{eq:Dirichlet-T} shows that we can establish a relationship between $\mathcal{E}_{\hat{T}}(f)$ and $\mathcal{E}_{T}(f)$ by comparing $\mathcal{E}_{Q_{i,x_{-i}}}( f_{i,x_{-i}} - \Pi_{i,x_{-i}} f_{i,x_{-i}})$ and $\| f_{i,x_{-i}} - \Pi_{i,x_{-i}} f_{i,x_{-i}} \|_{\Pi_{i,x_{-i}}}^2$.
			
			By \eqref{ine:Qiyi-bound},
			\begin{equation} \nonumber
				\begin{aligned}
					c_1 \, \| f_{i,x_{-i}} - \Pi_{i,x_{-i}} f_{i,x_{-i}} \|_{\Pi_{i,x_{-i}}}^2 &\leq \mathcal{E}_{Q_{i,x_{-i}}}( f_{i,x_{-i}} - \Pi_{i,x_{-i}} f_{i,x_{-i}} ) \\
					& \leq c_2 \, \| f_{i,x_{-i}} - \Pi_{i,x_{-i}} f_{i,x_{-i}} \|_{\Pi_{i,x_{-i}}}^2 .
				\end{aligned}
			\end{equation}
			Combining this with \eqref{eq:Thatff-1} and \eqref{eq:Dirichlet-T} gives the desired bound.
			
		\end{proof}
	}
	
	\begin{remark} \label{rem:exact}
		In hybrid Gibbs samplers, it is often the case that some conditional distributions are tractable, allowing for exact sampling from these distributions. 
		Mathematically, this means that, for some $i \in \{1,\dots,n\}$ and $y \in \X_{-i}$, one may take $Q_{i,y}(z, \cdot) = \Pi_{i,y}(\cdot)$ for each $z \in \X_i$.
		In this case, $\|Q_{i,y}\|_{\Pi_{i,y}} = 0$, and $\mathcal{E}_{Q_{i,y}}(f) = \|f\|_{\Pi_{i,y}}^2$ for $f \in L_0^2(\Pi_{i,y})$.
	\end{remark}
	
	Based on Theorem~\ref{thm:random}, one can immediately derive the following corollary, which
	gives the precise form of~\eqref{ine:Tbound-loose}.
	
	\begin{corollary} \label{cor:random}
		Suppose that there exists a constant $C \in [0,1]$ such that $\|Q_{i,y}\|_{\Pi_{i,y}} \leq C$ for $i = 1,\dots,n$ and $M_{-i}$-a.e. $y \in \X_{-i}$.
		Then
		\begin{equation} \label{ine:random-bound-1}
			(1-C) (1 - \|T\|_{\Pi}) \leq 1 - \|\hat{T}\|_{\Pi} \leq (1+C)(1 - \|T\|_{\Pi}).
		\end{equation}
		If, furthermore, $Q_{i,y}$ is positive semi-definite for $i = 1,\dots,n$ and $M_{-i}$-a.e. $y \in \X_{-i}$,
		then 
		\begin{equation} \label{ine:random-bound-2}
			1 - \|\hat{T}\|_{\Pi} \leq 1 - \|T\|_{\Pi}.
		\end{equation}
	\end{corollary}
	
	\begin{proof}
		Recall that $T$ is positive semi-definite.
		This implies that
		\[
		1 -\|T\|_{\Pi} = \inf_{f \in L_0^2(\Pi) \setminus \{0\}} \frac{\mathcal{E}_{T}( f)}{\|f\|_{\Pi}^2} \leq \sup_{f \in L_0^2(\Pi) \setminus \{0\}} \frac{\mathcal{E}_{T}( f)}{\|f\|_{\Pi}^2} \leq 1.
		\]
		See Section \ref{sec:preliminaries}.
		Moreover, $\|Q_{i,y}\|_{\Pi_{i,y}} \leq C$ implies that \eqref{ine:Qiyi-bound} holds with $c_1 = 1-C$ and $c_2 = 1+C$.
		Then \eqref{ine:random-bound-1} follows directly from Theorem \ref{thm:random} and \eqref{eq:norm-def}.
		When $Q_{i,y}$ is positive semi-definite for $i = 1,\dots,n$ and $M_{-i}$-a.e. $y \in \X_{-i}$, \eqref{ine:Qiyi-bound} holds with $c_1 = 1-C$ and $c_2 = 1$.
		Then \eqref{ine:random-bound-2} follows from Theorem \ref{thm:random} and \eqref{eq:norm-def}.
	\end{proof}

	{
		One may also use Theorem \ref{thm:random} to compare $T$ and $\hat{T}$ in terms of asymptotic variance.
		\begin{corollary} \label{cor:random-var}
			Suppose that there exist constants $c_1, c_2 \in (0,2)$ such that \eqref{ine:Qiyi-bound} holds for $i = 1,\dots,n$, $M_{-i}$-a.e. $y \in \X_{-i}$, and $g \in L_0^2(\Pi_{i,y}) \setminus \{0\}$.
			Assume further that $\|T\|_{\Pi} < 1$.
			Then, for $f \in L_0^2(\Pi)$,
			\[
			c_2^{-1} \mbox{var}_T(f) + (c_2^{-1} - 1) \|f\|_{\Pi}^2 \leq \mbox{var}_{\hat{T}}(f) \leq c_1^{-1} \mbox{var}_T(f) + (c_1^{-1} - 1) \|f\|_{\Pi}^2.
			\]
		\end{corollary}
	}
	
	\begin{proof}
		This follows from Theorem \ref{thm:random} and Theorem A.2 of \cite{caracciolo1990nonlocal}.
		See also Lemma 33 of \cite{andrieu18supplement}.
	\end{proof}
	
	
	\section{Data augmentation chain} \label{sec:da}
	
	Building on the work of \citet{andrieu2018uniform}, this section will present some new results regarding the relation between exact and hybrid data augmentation samplers.
	
	Recall that $S(y, \df y') = \int_{\X_2} \Pi_{2,y}(\df z) \, \Pi_{1,z}(\df y')$ is the Mtk of a data augmentation chain.
	The chain is reversible with respect to~$M_1$, and its convergence rate is determined by $\|S\|_{M_1}$.
	In fact, it can be shown that $S$ is positive semi-definite \citep{liu1995covariance}.
	For $z \in \X_2$, $Q_{1,z}$ is reversible with respect to $\Pi_{1,z}$, and the corresponding chain's convergence rate is $\|Q_{1,z}\|_{\Pi_{1,z}}$.
	The hybrid chain of interest is associated with the Mtk $\hat{S}(y, \df y') = \int_{\X_2} \Pi_{2,y}( \df z) \, Q_{1, z}(y, \df y')$.
	The chain is reversible with respect to~$M_1$ \cite{jones2014convergence}, and its convergence rate is determined by $\|\hat{S}\|_{M_1}$.

	For $M_2$-a.e. $z \in \X_2$, a function $f \in L_0^2(M_1)$ can be seen as an element in $L^2(\Pi_{1,z})$.
	This is because
	\[
	\int_{\X_1} f(y)^2 \, M_1(\df y) = \int_{\X_2} M_2(\df z) \int_{\X_1} \Pi_{1, z}(\df y) \, f(y)^2. 
	\]
	{
		The following lemma was derived in \cite{andrieu2018uniform}.
		
		\begin{lemma} \citep[][Section 7]{andrieu2018uniform} \label{lem:andrieu-0} 
			For $f \in L_0^2(M_1)$,
			\[
			\mathcal{E}_{S}(f) = \int_{\X_2} M_2(\df z) \|f - \Pi_{1,z} f \|_{\Pi_{1,z}}^2, \quad \mathcal{E}_{\hat{S}}( f) = \int_{\X_2} M_2(\df z) \, \mathcal{E}_{Q_{1,z}}( f).
			\]
		\end{lemma}
	}
	
	Immediately, one can obtain the precise version of \eqref{ine:Sbound-loose}.
	
	\begin{proposition} { \citep[][Theorem 25, Remark 26]{andrieu2018uniform}} \label{pro:andrieu-0}
		Suppose that there exist constants $c_1, c_2 \in [0,2]$ such that
		\[
		c_1 \leq \frac{\mathcal{E}_{Q_{1,z}}(g) }{\|g\|_{\Pi_{1,z}}^2 } \leq c_2
		\]
		for $M_2$-a.e. $z \in \X_2$ and $g \in L_0^2(\Pi_{1,z}) \setminus \{0\}$.
		Then, for $f \in L_0^2(\Pi)$,
		\[
		c_1 \, \mathcal{E}_{S}( f) \leq \mathcal{E}_{\hat{S}}(f) \leq c_2 \, \mathcal{E}_{S}(f).
		\]
		In particular, if there exists a constant $C \in [0,1]$ such that $\|Q_{1,z}\|_{\Pi_{1,z}} \leq C$ for $M_2$-a.e. $z \in \X_2$, then
		\[
		(1 - C) (1 - \|S\|_{M_1}) \leq 1 - \|\hat{S}\|_{M_1} \leq (1+C) (1 - \|S\|_{M_1}).
		\]
		If, furthermore, $Q_{1,z}$ is positive semi-definite for $M_2$-a.e. $z \in \X_2$, then the upper bound can be tightened to $1 - \|S\|_{M_1}$.
		A bound regarding asymptotic variance also exists.
	\end{proposition}

	Note the intriguing similarity between Theorem~\ref{thm:random} and Proposition \ref{pro:andrieu-0}, and among some existing spectral gap bounds outside the context of hybrid Gibbs samplers \citep{caracciolo1992two,madras1999importance,madras2002markov,carlen2003determination,qin2023geometric}.
	{
		The idea of decomposing Dirichlet forms seems to be present in most of these works, albeit not always transparently. }
	
	One undesirable feature of Proposition \ref{pro:andrieu-0} is that the lower bound on $1 - \|\hat{S}\|_{M_1}$ is nontrivial (i.e., strictly greater than~0) only if $\|Q_{1,z}\|_{\Pi_{1,z}}$ is uniformly bounded away from unity.
	The main goal of Section \ref{sec:da} is to relax this requirement.
	It turns out that we can achieve this through an argument from \cite{latuszynski2014convergence}, who studied hybrid slice samplers.
	(We were not able to make a meaningful improvement on Theorem \ref{thm:random} and Corollary \ref{cor:random}, which suffer from a similar restriction, through an analogous treatment.)
	
	The following theorem is proved in Appendix~\ref{app:deterministic-2}.
	{ The proof combines Lemma \ref{lem:andrieu-0}, which is taken from \cite{andrieu2018uniform}, with arguments from \cite{latuszynski2014convergence}.
	}

	\begin{theorem} \label{thm:deterministic-2}
		Let $\gamma: \X_2 \to [0,1]$ be a measurable function such that $\|Q_{1,z}\|_{\Pi_{1,z}} \leq \gamma(z)$ for $M_2$-a.e. $z \in \X_2$.
		Let~$t$ be some positive integer, and let $\alpha_t \in [0,1]$ be such that
		\[
		\int_{\X_2} \gamma(z)^t \, \Pi_{2,y}( \df z) \leq \alpha_t
		\]
		for $M_1$-a.e. $y \in \X_1$.
		If either (i) $t$ is even, or (ii) $Q_{1,z}$ is positive semi-definite for $M_2$-a.e. $z \in \X_2$, then, for $f \in L_0^2(M_1) \setminus \{0\}$,
		\begin{equation} \label{ine:Shat-Dir-2}
			0 \leq \frac{\langle f, \hat{S} f \rangle_{M_1}^t}{\|f\|_{M_1}^{2t}} \leq \frac{\langle f, S f \rangle_{M_1}}{\|f\|_{M_1}^2} + \alpha_t.
		\end{equation}
	\end{theorem}
	
	Theorem~\ref{thm:deterministic-2} gives the following bound regarding spectral gaps, which gives the precise form of \eqref{ine:Sbound-loose-2}.
	
	\begin{corollary} \label{cor:deterministic-2}
		For each positive integer~$t$, define $\alpha_t$ as in Theorem~\ref{thm:deterministic-2}.
		Then, if~$t$ is even, or if $Q_{1,z}$ is positive semi-definite for $M_2$-a.e. $z \in \X_2$,
		\[
		t(1-\|\hat{S}\|_{M_1}) \geq 1-\|\hat{S}\|_{M_1}^t \geq 1 - \|S\|_{M_1} - \alpha_t.
		\]
	\end{corollary}
	
	\begin{proof}
		The first inequality is the Bernoulli inequality $t(1-x) \geq 1 - x^t$ which holds for $x \in [-1,1]$.
		The second inequality follows from Theorem~\ref{thm:deterministic-2} and~\eqref{eq:norm-def}.
	\end{proof}
	
	\begin{remark}
		Corollary~\ref{cor:deterministic-2} can be regarded as a direct generalization of Theorem~8 of \cite{latuszynski2014convergence}.
		The latter gives a similar bound for a hybrid slice sampler.
		See Section~\ref{ssec:sli} for more details.
	\end{remark}
	
	\begin{remark}
		As mentioned in Section \ref{ssec:loose}, when $\|S\|_{M_1} < 1$, for the lower bound on $1 - \|\hat{S}\|_{M_1}^t$ in Corollary \ref{cor:deterministic-2} to be nontrivial for large values of~$t$, it suffices to have $\alpha_t \to 0$ as $t \to \infty$.
		
	\end{remark}
	
	\begin{remark}
		An immediate consequence of Proposition \ref{pro:andrieu-0} and Corollary \ref{cor:deterministic-2} is: Suppose that $\alpha_t \to 0$ as $t \to \infty$.
		Then $\hat{S}$ admits a strictly positive spectral gap if and only if $S$ does.
	\end{remark}

	One may also use Theorem \ref{thm:deterministic-2} to obtain the following result regarding asymptotic variance.
	{
		\begin{corollary}
			Suppose that there is a positive integer $t$ such that $\alpha_t \leq (1-\|S\|_{M_1})/2$, where $\alpha_t$ is defined in Theorem \ref{thm:deterministic-2}.
			Assume further that $t$ is even, or $S$ is positive semi-definite.
			Finally, assume that $\|S\|_{M_1} < 1$.
			Then, for $f \in L_0^2(M_1)$,
			\[
			\mbox{var}_{\hat{S}}(f) \leq 2t \, \mbox{var}_S(f) + (2t-1) \|f\|_{M_1}^2.
			\]
		\end{corollary}
	}
	
	\begin{proof}
		Let $f \in L_0^2 \setminus \{0\}$ be arbitrary.
		By Theorem \ref{thm:deterministic-2} and the Bernoulli inequality, 
		\[
		t \left( 1 - \frac{\langle f, \hat{S} f \rangle_{M_1}}{\|f\|_{M_1}^2} \right) \geq 1 - \frac{\langle f, S f \rangle_{M_1}}{\|f\|_{M_1}^2} - \alpha_t = \frac{\langle f, [(1-\alpha_t) I - S] f \rangle_{M_1} }{\|f\|_{M_1}^2} .
		\]
		Since $\alpha_t \leq (1 - \|S\|_{M_1})/2$, the operator $(1-2\alpha_t) I - S$ is positive semi-definite.
		Then
		\[
		\langle f, [(1-\alpha_t) I - S] f \rangle_{M_1} \geq \langle f, ( I - S ) f \rangle_{M_1} /2.
		\]
		Combining terms, we have
		\[
		\langle f, (I - \hat{S}) f \rangle_{M_1} \geq \frac{\langle f, (I-S) f \rangle_{M_1}}{2t}.
		\]
		Since $f$ is arbitrary and $\|S\|_{M_1} < 1$, the desired result then follows from Theorem A.2 of \cite{caracciolo1990nonlocal} or Lemma 33 of \cite{andrieu18supplement}.
	\end{proof}

	\section{Examples} \label{sec:examples}

	\subsection{Random-walk Metropolis within Gibbs} \label{ssec:met}
	
	In our first example, we compare a random-scan Gibbs sampler to its hybrid version with conditional distributions approximated by a random-walk Metropolis step.
	We first study a concrete example involving Bayesian regression with a spike-and-slab prior.
	Then we discuss a generic scenario where the conditional distributions are in some sense well-conditioned.
	
	{
		\subsubsection{Bayesian regression with a spike-and-slab prior}
		
		Assume that $\ell(\beta)$ is a log likelihood, where $\beta = (\beta_1, \dots, \beta_d) \in \mathbb{R}^d$ is a regression coefficient.
		Impose a continuous spike and slab prior \citep{george1993variable} on $\beta$.
		To be specific, we assume that 
		\[
		\begin{aligned}
			\beta_j \mid z, q &\sim \mbox{N}(0, z_j b_d + (1-z_j) a_d ) \; \text{ independently}, \\
			z_j \mid q & \sim \mbox{Bernoulli}(q)  \; \text{ independently}, \\
			q & \sim F_d,
		\end{aligned}
		\]
		where $a_d$ and $b_d$ are positive hyperparameters such that $a_d \ll b_d$, and $F_d$ is some distribution on $(0,1)$ with density function $f_d$.
		The resultant posterior distribution, which we denote by $\Pi$, has density function proportional to
		\[
		e^{\ell(\beta)} \left[ \prod_{j=1}^d \frac{1}{\sqrt{\tau(z_j)}} \exp \left( - \frac{\beta_j^2}{2\tau(z_j)} \right) q^{z_j} (1-q)^{1-z_j} \right]  f_d(q)
		\]
		for $\beta \in \X_1 := \mathbb{R}^d$, $z \in \X_2 := \{0,1\}^d$, $q \in \X_3 := (0,1)$, where $\tau(z_j) = z_j b_d + (1-z_j) a_d$.
		
		For tractability, assume that $\ell$ is convex and $L_d$-smooth, where $L_d > 0$.
		To be concrete, assume the following:
		\begin{itemize}
			\item [\namedlabel{H1}{(H1)}] The function $\ell$ is continuously twice-differentiable.
			The Hessian matrix of $\ell(\beta)$, $H(\beta)$, satisfies $O_d \preccurlyeq -H(\beta) \preccurlyeq L_d I_d$ for each $\beta \in \mathbb{R}^d$, where $O_d$ is the $d \times d$ zero matrix, $I_d$ is the $d \times d$ identity matrix, and, for two symmetric matrices $A_1$ and $A_2$, $A_1 \preccurlyeq A_2$ means that $A_2 - A_1$ is positive semi-definite.
		\end{itemize}
		Note that \ref{H1} is satisfied if $\ell(\beta)$ corresponds to a linear, logistic, or probit regression model \cite{pratt1981concavity}.

		Consider a random-scan Gibbs sampler targeting $\Pi$.
		There are three full conditional distributions.
		We will let $\Pi_{1, (z,q)}(\cdot)$ be the conditional distribution of $\beta$ given $(z,q)$, and define $\Pi_{2, (\beta, q)}$ and $\Pi_{3, (\beta, z)}$ analogously.
		The second and third full conditional distributions may be sampled from efficiently using rejection sampling, since they are either univariate or a product of univariate distributions.
		On the other hand, it may be costly to sample exactly from $\Pi_{1, (z,q)}(\cdot)$.
		A random walk Metropolis-Hastings step may be used to approximate $\Pi_{1, (z,q)}(\cdot)$, as we now describe:
		\begin{itemize}
			\item Given the current state $\beta$, independently draw $\beta'_j$, $j \in \{1,\dots,d\}$, from the $\mbox{N}(\beta_j, \tau(z_j) \sigma_d^2 )$ distribution, where $\sigma_d > 0$.
			\item The proposed state $\beta'$ is then accepted or rejected based on a Metropolis-Hastings step, ensuring that the transition law is reversible with respect to $\Pi_{1, (z,q)}(\cdot)$.
		\end{itemize}
		Let $Q_{1,(z,q)}$ be the corresponding transition kernel.
		Let $T$ be the Mtk of an exact random-scan Gibbs sampler, i.e.,
		\[
		\begin{aligned}
			T((\beta, z, q), \df (\beta', z', q')) =& p_1 \, \Pi_{1,(z,q)}(\df \beta') \, \delta_{(z,q)}(\df (z',q')) + p_2 \, \Pi_{2, (\beta,q)}(\df z') \, \delta_{(\beta, q)}(\df (\beta',q')) + \\
			& p_3 \,  \Pi_{3, (\beta,z)}(\df q') \, \delta_{(\beta,z)}(\df (\beta',z')),
		\end{aligned}
		\]
		where $(p_1, p_2, p_3)$ is a probability vector.
		Let $\hat{T}$ be the Mtk of the corresponding hybrid Gibbs sampler, with $\Pi_{1,(z,q)}(\df \beta')$ replaced by $Q_{1,(z,q)}(\beta, \df \beta')$.
		Our goal for this example is to place a bound on the ratio $(1-\|\hat{T}\|_{\Pi})/(1-\|T\|_{\Pi})$.

		The following proposition gives a bound on $\|Q_{1,(z,q)}\|_{\Pi_{1,(z,q)}}$.
		The proof, presented in Appendix \ref{app:metro}, is an application of the main result of \citet{andrieu2022explicit}.

		\begin{proposition} \label{pro:metro}
			Suppose that \ref{H1} holds, and that $\sigma_d^2 = 1/[2d (b_d L_d + 1)]$.
			Then
			\[
			\frac{C_0}{d ( b_d L_d + 1)} \leq \frac{\mathcal{E}_{Q_{1,(z,q)}}(g)  }{\|g\|_{\Pi_{1,(z,q)}}^2} \leq 1
			\]
			for $z \in \{0,1\}^d$, $q \in (0,1)$, and $g \in L_0^2(\Pi_{1,(z,q)}) \setminus \{0\}$, where $C_0$ is some positive universal constant.
		\end{proposition}

		Then, by Corollaries \ref{cor:random} and \ref{cor:random-var}, along with Remark \ref{rem:exact}, we have the following bounds.
		\begin{corollary} \label{cor:metro}
			Suppose that \ref{H1} holds, and that $\sigma_d^2 = 1/[2d (b_d L_d + 1)]$.
			Then
			\[
			\frac{C_0}{d ( b_d L_d + 1)} (1-\|T\|_{\Pi}) \leq 1 - \|\hat{T}\|_{\Pi} \leq 1 - \|T\|_{\Pi},
			\]
			\[
			\mbox{var}_T(f) \leq \mbox{var}_{\hat{T}}(f) \leq \frac{d (b_d L_d + 1)}{C_0} \mbox{var}_T(f) + \frac{d(b_d L_d + 1) - C_0}{C_0} \|f\|_{\Pi}^2
			\]
			for $f \in L_0^2(\Pi)$, where $C_0$ is a positive universal constant.
		\end{corollary}
		
		Since $T$ is positive semi-definite, it holds that $\mbox{var}_T(f) \geq \|f\|_{\Pi}$ for $f \in L_0^2(\Pi)$.
		Then Corollary \ref{cor:metro} implies that, when $d \to \infty$, $\mbox{var}_{\hat{T}}(f)/ \mbox{var}_T(f) = O(db_d L_d + d)$.
		It can be shown that, when $\ell(\beta)$ corresponds to a standard probit or logistic regression model with design matrix $W \in \mathbb{R}^{N_d \times d}$, one may take $L_d = \lambda_{\scriptsize\max}(W^{\top} W)$, where $\lambda_{\scriptsize\max}$ returns the largest eigenvalue of a matrix \cite{bohning1999lower,demidenko2001computational}.
		To estimate the magnitude of $\mbox{var}_{\hat{T}}(f)/ \mbox{var}_T(f)$, we take $b_d = O(1)$, and it is reasonable to assume that $L_d = O(N_d + d)$ \citep{karoui2003largest,johnstone2001distribution}.
		Under these assumptions, $\mbox{var}_{\hat{T}}(f)/ \mbox{var}_T(f) = O(d N_d + d^2)$.

		In cases where it is feasible but costly to implement the exact Gibbs sampler,
		the ratio $\mbox{var}_{\hat{T}}(f)/ \mbox{var}_T(f)$ can be compared to the ratio of the average per-iteration costs of $T$ and $\hat{T}$ to determine the more efficient sampler.
		For illustration, suppose that $\ell(\beta)$ corresponds to a probit regression model based on $N_d$ iid observations, where the cost of evaluating $\ell(\beta)$ scales as $d N_d$.
		Then the cost of drawing from $Q_{1,(z,q)}(\beta, \cdot)$ is on the order of $d N_d$.
		On the other hand, some algorithms allow exact sampling from $\Pi_{1,(z,q)}$ \citep{botev2017normal,durante2019conjugate}, such as Algorithm~1 in \cite{durante2019conjugate}.
		The per-iteration cost of this algorithm after preprocessing appears to scale as $d N_d + d^2$, plus the cost of sampling from an $N_d$-dimensional truncated normal distribution with $N_d$ linear constraints.
		Sampling from the truncated normal distributions can be performed using a rejection sampler, but the cost, denoted by $C(N_d)$, may be prohibitively large for large $N_d$ \cite{chopin2011fast,botev2017normal}.
		The costs of updating $z$ and $q$ are assumed to be negligible compared to the cost of drawing from $Q_{1,(z,q)}(\beta, \cdot)$ or $\Pi_{1,(z,q)}(\cdot)$.
		Thus, the ratio of the average per-iteration costs of $T$ and $\hat{T}$ is $\Omega(1 + d N_d^{-1} + d^{-1} N_d^{-1} C(N_d))$.
		Comparing this to the asymptotic order of $\mbox{var}_{\hat{T}}(f)/ \mbox{var}_T(f)$ in the previous paragraph shows that the Metropolis-within-Gibbs algorithm is more efficient if 
		\[
		C(N_d) \gg d^2 N_d (N_d + d).
		\]
	}
	
	\subsubsection{A generic setting} \label{sssec:generic}
	
	More broadly speaking, Assume that $\Pi$ is a distribution on $\X = \X_1 \times \cdots \times \X_n$, and that $\X_i = \mathbb{R}^{d_i}$, where each $d_i$ is a positive integer.
	Assume that $\Pi$ has a density function proportional to $x \mapsto e^{-\xi(x)}$ for some function $\xi: \X\to \mathbb{R}$.
	For $i \in \{1,\dots,n\}$ and $y \in \X_{-i}$, define $\xi_{i,y}: \mathbb{R}^{d_i} \to \mathbb{R}$ to be such that $\xi_{i,x_{-i}}(x_i) = \xi(x)$ for $x \in \X$.
	Then $z \mapsto \xi_{i,y}(z)$ gives the log density function of the conditional distribution $\Pi_{i,y}$.
	
	We now consider exact and hybrid random-scan Gibbs samplers targeting~$\Pi$.
	We shall make the following assumptions regarding the conditional distributions.
	\begin{itemize}
		\item [\namedlabel{H2}{(H2)}] 
		\begin{enumerate}
			\item There exists $k \in \{1,\dots,d\}$ such that, for $i \leq k$ and $y \in \X_{-i}$, the log density $\xi_{i,y}$ is well conditioned in the following sense:
			there exist positive numbers $m_{i,y}$ and $L_{i,y}$ such that, for $z, u \in \mathbb{R}^{d_i}$,
			\begin{equation} \nonumber
				\frac{m_{i,y}}{2} \|u\|^2 \leq \xi_{i,y}(z + u) - \xi_{i,y}(z) - \langle \nabla \xi_{i,y}(z), u \rangle \leq \frac{L_{i,y}}{2} \|u\|^2,
			\end{equation}
			where $\|\cdot\|$ is the Euclidean norm, and $\langle u, v \rangle = u^{\top} v$ for $u, v \in \mathbb{R}^{d_i}$;
			in other words, $\xi_{i,y}(\cdot)$ is $m_{i,y}$-strongly convex and $L_{i,y}$-smooth.
			\item For $i > k$ and $y \in \X_{-i}$ , one can make exact draws from $\Pi_{i,y}$ efficiently.
			Thus, these conditional distributions are not approximated even in the hybrid sampler.
		\end{enumerate}
	\end{itemize}
	In the hybrid algorithm, we use the following random-walk Metropolis algorithm to approximate $\Pi_{i,y}$ when $i \leq k$.
	Given the current state $z \in \mathbb{R}^{d_i}$, propose a new state by drawing from the normal distribution with mean~$z$ and covariance matrix $\sigma_{i,y}^2 I_{d_i}$, where $\sigma_{i,y} > 0$, then accept or reject the proposed state through a Metropolis step.
	Let $Q_{i,y}$ be the Mtk of this random-walk Metropolis algorithm.
	
	Theorem~1 of \cite{andrieu2022explicit} along with Lemma 3.1 of \cite{baxendale2005renewal} implies the following:
	\begin{proposition} \label{pro:example-random}
		Suppose that \ref{H2} holds.
		Let $i \in \{1,\dots,k\}$ and let $y \in \X_{-i}$.
		Set $\sigma_{i,y}^2 = 2^{-1} L_{i,y}^{-1} d_i^{-1}$.
		Then
		\[
		\frac{C_0 m_{i,y}}{d_i L_{i,y}} \leq \frac{\mathcal{E}_{Q_{i,y}}(g)}{\|g\|_{\Pi_{i,y}}^2 } \leq 1
		\]
		for $g \in L_0^2(\Pi_{i,y}) \setminus \{0\}$, where $C_0$ is a universal constant.
	\end{proposition}
	
	Define $T$ and $\hat{T}$ as in \eqref{eq:T} and \eqref{eq:That}.
	Since one can  sample from $\Pi_{i,y}$ efficiently for each $y \in \X_{-i}$ when $i > k$, for these values of $(i,y)$, $Q_{i,y}(z, \cdot)$ is set to $\Pi_{i,y}(\cdot)$ for $z \in \mathbb{R}^{d_i}$.
	Using Corollary \ref{cor:random} along with Proposition \ref{pro:example-random} and recalling Remark \ref{rem:exact}, we derive the following.
	\begin{proposition}
		Suppose that \ref{H2} holds.
		For $i \in \{1,\dots,k\}$ and $y \in \X_{-i}$, set $\sigma_{i,y}^2 = 2^{-1} L_{i,y}^{-1} d_i^{-1}$.
		Then
		\[
		\left( \min_{i \leq k} \inf_{y \in \X_{-i}} \frac{C_0 m_{i,y}}{d_i L_{i,y}} \right) (1 - \|T\|_{\Pi}) \leq 1 - \|\hat{T}\|_{\Pi} \leq 1 - \|T\|_{\Pi}.
		\]
	\end{proposition}

	\subsection{Random-scan Gibbs samplers with blocked updates} \label{ssec:hie}
	
	When using a random-scan Gibbs sampler to draw from an $n$-component distribution $\Pi$, sometimes it may be possible to update multiple components using their joint conditional distribution at once.
	To define such an algorithm, we need some notations.
	For $\Lambda \subset [n] := \{1,\dots,n\}$ such that $1 \leq |\Lambda| \leq n-1$ and $(x_1, \dots, x_n) \in \X$, let $x_{\Lambda} = (x_i)_{i \in \Lambda}$, and let $x_{-\Lambda} = (x_i)_{i \not\in \Lambda}$.
	Let $M_{-\Lambda}$ be the marginal distribution of $X_{-\Lambda}$ where $X \sim \Pi$, and let $\Pi_{\Lambda, y}$ be the conditional distribution of $X_{\Lambda}$ given $X_{-\Lambda} = y$ for $y \in \X_{-\Lambda} = \{x_{-\Lambda}: \, x \in \X\}$.
	Consider a random-scan Gibbs sampler that updates $\ell$ components in one iteration, where $\ell \in [n-1]$.
	To be specific, in each iteration, given the current state $x$, a subset $\Lambda \subset [n]$ whose cardinality is $\ell$ is randomly and uniformly selected, and $x_{\Lambda}$ is updated according to the conditional distribution $\Pi_{\Lambda, x_{-\Lambda}}$.
	The corresponding Mtk is
	\[
	T_{\ell}(x, \df x') = \frac{1}{{n \choose \ell}} \sum_{\Lambda \subset [n]} \ind(|\Lambda| = \ell) \,  \Pi_{\Lambda, x_{-\Lambda}}(\df x'_{\Lambda}) \, \delta_{x_{-\Lambda}} (\df x'_{-\Lambda}),
	\]
	where $\ind(|\Lambda| = \ell)$ is 1 if $|\Lambda| = \ell$ and 0 otherwise.
	One can check that $T_{\ell}$ defines a positive semi-definite operator on $L_0^2(\Pi)$.
	Let $\ell$ and $\emm$ be integers such that $1 \leq \emm < \ell \leq n-1$.
	A natural question is whether there is a quantitative relationship between the spectral gaps of $T_{\ell}$ and $T_{\emm}$.
	
	To give a partial answer to the question, we demonstrate that $T_{\emm}$ can be seen as a hybrid version of $T_{\ell}$, and the Markovian approximation step corresponds to yet another random-scan Gibbs sampler.
	(This observation is not new, and can be traced back to at least \cite{roberts1997geometric}.)
	For $\Lambda \subset [n]$ such that $|\Lambda| = \ell$ and $y \in \X_{-\Lambda}$, let $Q_{\Lambda, y}$ be an Mtk on $\X_{\Lambda}$ satisfying
	\[
	Q_{\Lambda, x_{-\Lambda}}(x_{\Lambda}, \df x_{\Lambda}') = \frac{1}{{\ell \choose \emm}} \sum_{\Gamma \subset \Lambda} \ind(|\Gamma| = \emm) \, \Pi_{\Gamma, x_{-\Gamma}}(\df x_{\Gamma}') \, \delta_{x_{\Lambda \setminus \Gamma}}(\df x_{\Lambda \setminus \Gamma}').
	\]
	To sample from $Q_{\Lambda, x_{-\Lambda}}(x_{\Lambda}, \cdot)$, one randomly and uniformly selects a subset $\Gamma \subset \Lambda$ such that $|\Gamma| = \emm$, and draw from the conditional distribution of $X_{\Gamma}$ given $X_{-\Gamma} = x_{-\Gamma}$, i.e., given $X_{-\Lambda} = x_{-\Lambda}$ and $X_{\Lambda \setminus \Gamma} = x_{\Lambda \setminus \Gamma}$.
	Observe that, for $M_{-\Lambda}$-a.e. $y \in \X_{-\Lambda}$, $Q_{\Lambda, y}$ can be viewed as the Mtk of a random-scan Gibbs sampler targeting the $\ell$-component distribution $\Pi_{\Lambda, y}$ that updates $\emm$ components in each step.
	In particular, $Q_{\Lambda, y}$ can be viewed as a positive semi-definite operator on $L_0^2(\Pi_{\Lambda, y})$.
	
	It is straightforward to derive that
	\[
	\begin{aligned}
		T_{\emm}(x, \df x') &= \frac{1}{{n \choose \emm}} \sum_{\Gamma \subset [n]} \ind(|\Gamma| = \emm) \,  \Pi_{\Gamma, x_{-\Gamma}}(\df x'_{\Gamma}) \, \delta_{x_{-\Gamma}} (\df x'_{-\Gamma}) \\
		&= \frac{1}{{n \choose \ell}} \sum_{\Lambda \subset [n]} \ind(|\Lambda| = \ell) \, \frac{1}{{\ell \choose \emm}} \sum_{\Gamma \subset \Lambda}  \ind(|\Gamma| = \emm) \, \Pi_{\Gamma, x_{-\Gamma}}(\df x_{\Gamma}') \, \delta_{x_{-\Gamma}} (\df x'_{-\Gamma}) \\
		&= \frac{1}{{n \choose \ell}} \sum_{\Lambda \subset [n]} \ind(|\Lambda| = \ell) \, Q_{\Lambda, x_{-\Lambda}}(x_{\Lambda}, \df x_{\Lambda}') \, \delta_{x_{-\Lambda}}(\df x_{-\Lambda}').
	\end{aligned}
	\]
	Note that $T_{\emm}(x, \df x')$ is $T_{\ell}(x, \df x')$ with the conditional distribution $\Pi_{\Lambda, x_{-\Lambda}}(\df x'_{\Lambda})$ replaced by the Markovian approximation $Q_{\Lambda, x_{-\Lambda}}(x_{\Lambda}, \df x_{\Lambda}')$.
	It is in this sense that $T_{\emm}$ is a hybrid version of $T_{\ell}$.
	
	In the same way that Theorem \ref{thm:random} is proved, we can establish the following relationship between $T_{\ell}$ and $T_{\emm}$ in terms of the Markovian approximation $Q_{\Lambda, y}$.
	\begin{proposition} \label{pro:block}
		Suppose that there exist constants $c_1 \in [0,1]$ such that
		\begin{equation} \label{ine:Dirichlet-block}
			c_1 \leq \frac{\mathcal{E}_{Q_{\Lambda,y}}(g) }{\|g\|_{\Pi_{\Lambda,y}}^2 }
		\end{equation}
		for $\Lambda \subset [n]$ such that $|\Lambda| = \ell$, $M_{-\Lambda}$-a.e. $y \in \X_{-\Lambda}$, and $g \in L_0^2(\Pi_{\Lambda, y}) \setminus \{0\}$.
		Then, for $f \in L_0^2(\Pi)$,
		\[
		c_1 \mathcal{E}_{T_{\ell}}(f) \leq \mathcal{E}_{T_{\emm}}(f) \leq  \mathcal{E}_{T_{\ell}}(f).
		\]
	\end{proposition}
	
	\begin{proof}
		Let $f \in L_0^2(\Pi)$.
		For $\Lambda \subset [n]$ and $y \in \X_{-\Lambda}$, let $f_{\Lambda, y}: \X_{\Lambda} \to \mathbb{R}$ be such that $f_{\Lambda, x_{-\Lambda}}(x_{\Lambda}) = f(x)$ for $x \in \X$.
		Then, one can show that, for $M_{-\Lambda}$-a.e. $y \in \X_{-\Lambda}$, it holds that $f_{\Lambda, y} - \Pi_{\Lambda,y} f_{\Lambda, y} \in L_0^2(\Pi_{\Lambda, y})$.
		
		Analogous to \eqref{eq:Thatff-1} and \eqref{eq:Dirichlet-T},
		\[
		\begin{aligned}
			\mathcal{E}_{T_{\emm}}(f) &= \frac{1}{{n \choose \ell}} \sum_{\Lambda \subset [n]} \ind(|\Lambda| = \ell) \, \int_{\X_{-\Lambda}} M_{-\Lambda}(\df x_{-\Lambda}) \, \mathcal{E}_{Q_{\Lambda, x_{-\Lambda}}}(f_{\Lambda, x_{-\Lambda}} - \Pi_{\Lambda, x_{-\Lambda}} f_{\Lambda, x_{-\Lambda}}), \\
			\mathcal{E}_{T_{\ell}}(f) &= \frac{1}{{n \choose \ell}} \sum_{\Lambda \subset [n]} \ind(|\Lambda| = \ell) \, \int_{\X_{-\Lambda}} M_{-\Lambda}(\df x_{-\Lambda}) \, \| f_{\Lambda, x_{-\Lambda}} - \Pi_{\Lambda,x_{-\Lambda}} f_{\Lambda, x_{-\Lambda}} \|_{\Pi_{\Lambda, x_{-\Lambda}} }^2.
		\end{aligned}
		\]
		The desired result then follows from \eqref{ine:Dirichlet-block} and the fact that $Q_{\Lambda, y}$ is positive semi-definite.
	\end{proof}
	
	The following corollary is also easy to derive.
	\begin{corollary} \label{cor:block}
		Suppose that there exist constants $c_1 \in [0,1]$ such that \eqref{ine:Dirichlet-block} holds for $\Lambda \subset [n]$ such that $|\Lambda| = \ell$, $M_{-\Lambda}$-a.e. $y \in \X_{-\Lambda}$, and $f \in L_0^2(\Pi_{\Lambda, y}) \setminus \{0\}$.
		Then
		\[
		c_1 (1 - \|T_{\ell}\|_{\Pi}) \leq 1 - \|T_{\emm}\|_{\Pi} \leq 1 - \|T_{\ell}\|_{\Pi}.
		\]
		Moreover, if $\|T_{\ell}\|_{\Pi} < 1$ and $c_1 > 0$, then for $f \in L_0^2(\Pi)$, 
		\[
		\mbox{var}_{T_{\ell}}(f) \leq \mbox{var}_{T_{\emm}}(f) \leq c_1^{-1} \mbox{var}_{T_{\ell}}(f) + (c_1^{-1} - 1) \|f\|_{\Pi}^2.
		\]
	\end{corollary}
	
	Corollary \ref{cor:block} extends Theorem 2 of \cite{qin2023spectral}, which gives the relation for $\ell = n-1$.
	The latter was used to extend the spectral independence technique, a popular tool for analyzing the convergence properties of Glauber dynamics on discrete spaces \cite{anari2021spectral,feng2022rapid,chen2021rapid,jain2021spectral}, to general state spaces.

	\subsection{A hybrid slice sampler} \label{ssec:sli}
	
	In this subsection, we use Theorem \ref{thm:deterministic-2} to reproduce and slightly improve a recent development in the analysis of hybrid slice samplers.
	
	Consider the setting of the data augmentation algorithm.
	Let $\X_1$ be a non-empty open subset of $\mathbb{R}^d$, where $d$ is a positive integer.
	Suppose that the distribution of interest, $M_1$, has a possibly un-normalized density function $\tilde{m}: \X_1 \to (0,\infty)$ with respect the $d$-dimensional Borel measure.
	Let $\|\tilde{m}\|_{\infty}$ be the essential supremum of $\tilde{m}$ with respect to the Borel measure, and assume that it is finite.
	Set $\X_2 = (0, \|\tilde{m}\|_{\infty} )$, and, for $y \in \X_1$, let $\Pi_{2,y}$ be the uniform distribution on $(0, \tilde{m}(y))$.
	Then $\Pi$ is the uniform distribution on
	\[
	G = \{(y,z) \in \X_1 \times \X_2: \, z < \tilde{m}(y) \}.
	\]
	In this case, the distribution $M_2$ has density function proportional to
	\[
	v(z) = \int_{\X_1} \ind(y \in G_z) \, \df y \text{ where } G_z = \{y \in \X_1: \, z < \tilde{m}(y)\}, \quad z \in \X_2.
	\]
	Here, $G_z$ is called a "level set."
	For $z \in (0,\|\tilde{m}\|_{\infty})$, one can take $\Pi_{1,z}$ to be the uniform distribution on $G_z$.
	The data augmentation algorithm, which has Mtk $S(y, \df y') = \int_0^{\|\tilde{m}\|_{\infty}} \Pi_{2, y}(\df z) \, \Pi_{1,z}(\df y')$, then corresponds to the well-known slice sampler \citep[see, e.g.,][]{neal2003slice}.
	In each iteration, given the current state $y$, one draws $z$ from the uniform distribution on $(0, \|\tilde{m}\|_{\infty})$, and then draws the next state $y'$ from the uniform distribution on $G_z$.
	
	There have been a number of works on the convergence properties of the slice sampler \citep{roberts1999convergence,mira2002efficiency,natarovskii2021quantitative}.
	For instance, the following was established in \cite{natarovskii2021quantitative}.
	
	\begin{proposition} \citep[][Theorem 3.10, Proposition 3.13]{natarovskii2021quantitative}
		Suppose that $v$ is continuously differentiable on the bounded interval $(0, \|\tilde{m}\|_{\infty})$ with $v'(z) < 0$.
		Suppose further that, for some positive integer~$k$, the function $z \mapsto z v'(z) / v(z)^{1-1/k}$ is decreasing on $(0, \|\tilde{m}\|_{\infty})$.
		Then
		\[
		1 - \|S\|_{M_1} \geq \frac{1}{k+1}.
		\]
	\end{proposition}

	Unfortunately, sampling from $\Pi_{1,z}$, the uniform distribution on $G_z$, may be difficult in practice.
	Instead, one can simulate one iteration of the chain associated with an Mtk $Q_{1,z}$, where $\Pi_{1,z}$ is reversible with respect to $Q_{1,z}$.
	Call the Mtk of the resulting hybrid data augmentation sampler~$\hat{S}$.
	Then, by Theorem \ref{thm:deterministic-2}, we have the following result.
	
	\begin{proposition} \label{pro:slice}
		Let $\gamma: (0, \|\tilde{m}\|_{\infty} ) \to [0,1]$ be a measurable function such that $\|Q_{1,z}\|_{\Pi_{1,z}} \leq \gamma(z)$ for $M_2$-a.e. $z \in (0, \|\tilde{m}\|_{\infty} )$.
		Let~$t$ be some positive integer, and let $\alpha_t \in [0,1]$ be such that
		\[
		\frac{1}{\tilde{m}(y)} \int_0^{\tilde{m}(y)} \gamma(z)^t \, \df z \leq \alpha_t
		\]
		for $M_1$-a.e. $y \in \X_1$.
		If either $t$ is even, or $Q_{1,z}$ is positive semi-definite for $M_2$-a.e. $z \in (0, \|\tilde{m}\|_{\infty} )$, then, for $f \in L_0^2(M_1) \setminus \{0\}$,
		\[
		\frac{\langle f, \hat{S} f \rangle_{M_1}^t}{\|f\|_{M_1}^{2t}} \leq \frac{\langle f, S f \rangle_{M_1}}{\|f\|_{M_1}^2} + \alpha_t.
		\]
	\end{proposition}

	In particular, by Corollary \ref{cor:deterministic-2}, if $Q_{1,z}$ is almost surely positive semi-definite,
	\begin{equation} \label{ine:slice-ours}
		\frac{1 - \|S\|_{M_1} - \alpha_t}{t} \leq 1 - \|\hat{S}\|_{M_1} \leq 1 - \|S\|_{M_1}.
	\end{equation}
	(The upper bound follows from Proposition \ref{pro:andrieu-0}.)
	
	The above display recovers the main technical result of \cite{latuszynski2014convergence}.
	When $Q_{1,z}$ is almost surely positive semi-definite, Theorem~8 of \cite{latuszynski2014convergence} states the following:
	For each positive integer~$t$,
	\begin{equation} \label{ine:slice-previous}
		\frac{1 - \|S\|_{M_1} - \beta_t}{t} \leq 1 - \|\hat{S}\|_{M_1} \leq 1 - \|S\|_{M_1},
	\end{equation}
	where $\beta_t \in [0,1]$ satisfies
	\[
	\left[ \frac{1}{\tilde{m}(y)} \int_0^{\tilde{m}(y)}  \gamma(z)^{2t}  \, \df z \right]^{1/2} \leq \beta_t \; \; \text{for } M_1 \text{-a.e. } y.
	\]
	By Jensen's inequality, one may always take $\alpha_t \leq \beta_t$, so the lower bound in~\eqref{ine:slice-ours} is slightly sharper than that in~\eqref{ine:slice-previous}.
	
	{
		
		To ensure that the lower bound in \eqref{ine:slice-ours} is nontrivial for some~$t$ whenever $\|S\|_{M_1} < 1$, it suffices to have $\alpha_t \to 0$ as $t \to \infty$.
		The following result, which is proved in Appendix \ref{app:latu}, is an easy extension of Lemma 10 of \cite{latuszynski2014convergence}.
		\begin{proposition} \label{pro:latu}
			Suppose that, in Proposition \ref{pro:slice}, $\gamma(z) < 1$ almost everywhere on the bounded interval $(0,\|\tilde{m}\|_{\infty})$, and that $\gamma(z) \leq \gamma_0$ when $z \in (0,\epsilon]$ for some $\gamma_0 < 1$ and $\epsilon > 0$.
			Then, taking
			\[
			\alpha_t = \sup_{m \in (0, \|\tilde{m}\|_{\infty})} \frac{1}{m} \int_0^m \gamma(z)^t \, \df z, \quad
			\beta_t = \sup_{m \in (0, \|\tilde{m}\|_{\infty})}  \left[ \frac{1}{m} \int_0^m  \gamma(z)^{2t}  \, \df z \right]^{1/2}, 
			\]
			we have $\alpha_t \to 0$ and $\beta_t \to 0$ as $t \to \infty$.
		\end{proposition}

		A concrete choice for $Q_{1,z}$ is the Mtk of the hit-and-run sampler \citep{smith1984efficient} targeting the uniform distribution $\Pi_{1,z}$.
		In the hit-and-run sampler, starting from the current state $y \in \X_1$, a line passing through $y$ is selected uniformly at random; the next state is then chosen uniformly from the points where this line intersects $G_z$.
		The resultant $Q_{1,z}$ is positive semi-definite \citep{rudolf2013positivity}.
		There have been plenty of studies on the convergence properties of hit-and-run samplers \citep{diaconis1998hit,lovasz1999hit,lovasz2004hit}.
		For example, Theorem 4.2 of \cite{lovasz2004hit} along with Cheeger's inequality \cite{lawler1988bounds} give the following.
		\begin{proposition} \label{pro:lovasz}
			Let $z \in (0,\|\tilde{m}\|_{\infty})$.
			Suppose that $G_z$ is a convex subset of $\mathbb{R}^d$.
			Suppose further that $G_z$ contains a ball with radius $r_z$, while being contained in a ball with radius $R_z$, where $0 < r_z \leq R_z < \infty$.
			Then, when $Q_{1,z}$ is the Mtk of the hit-and-run sampler targeting $\Pi_{1,z}$,
			\[
			1 - \|Q_{1,z}\|_{\Pi_{1,z}} \geq \frac{C_1 r_z^2}{d^2 R_z^2 },
			\]
			where $C_1$ is a positive universal constant.
		\end{proposition}

		Assuming that $z \mapsto r_z/R_z$ is measurable in Proposition \ref{pro:lovasz}, we can take $\gamma(z) = 1 - (C_1/d^2)(r_z/R_z)^2$.
		By Proposition \ref{pro:latu}, in this case, to ensure that $\alpha_t$ converges to~0, it suffices to have $r_z/R_z$ being bounded away from 0 when $z \to 0$.
		That is, when $z \approx 0$, the level set $G_z$ is not excessively stretched and elongated.

		\begin{remark}
			When the level sets are not convex, there are other spectral gap bounds for the hit-and-run algorithm that allows for non-trivial choices of $\gamma(z)$. 
			See, e.g., Proposition~1 of \cite{latuszynski2014convergence}.
		\end{remark}
		
		For a quick example, let $\X_1 = \{y \in (0,1)^d: \, \sum_{i=1}^d y_i < 1 \}$, where $d \geq 2$.
		Let
		\[
		\tilde{m}(y) = \frac{1 - \sum_{i=1}^d y_i}{1 - y_d}.
		\]
		Then $\X_2 = (0,1)$, and, for $z \in (0,1)$, $G_z = \{ y \in (0,1)^d: \, (1-z)^{-1} \sum_{i=1}^{d-1} y_i + y_d < 1 \}$.
		Hence, $v(z) = (1-z)^{d-1}/d!$.
		It is straightforward to verify that $v'(z) < 0$, and $z \mapsto z v'(z)/v(z)^{1-1/k}$ is decreasing when $k = d-1$.
		So by Proposition \ref{pro:slice}, $1 - \|S\|_{M_1} \geq 1/d$.
		The convex level set $G_z$ contains the ball centered at $((1-z)/(2d), \dots, (1-z)/(2d))$ of radius $(1-z)/(2d)$, while being contained in the ball centered at the origin of radius~1.
		Then, by Proposition \ref{pro:lovasz}, when $Q_{1,z}$ corresponds to the hit-and-run algorithm, one may take $\gamma(z) = 1 - (C_1/4)(1-z)^2/d^4$, where $C_1$ is a positive universal constant.
		Note that $\gamma(z)$ is not uniformly bounded away from~1 when $z \in (0,1)$, so the lower bound for $1 - \|\hat{S}\|_{M_1}$ in Proposition \ref{pro:andrieu-0} does not apply here.
		Instead, in Proposition \ref{pro:slice}, we may take 
		\[
		\alpha_t = \sup_{m \in (0, 1)} \frac{1}{m} \int_0^m \gamma(z)^t \, \df z = \int_0^1 \left[ 1 - \frac{C_1 (1-z)^2}{4 d^4} \right]^t \, \df z,
		\]
		where the second equality holds since $\gamma(z)$ is increasing.
		By the dominated convergence theorem, when $t = d^{4+\epsilon}$, where $\epsilon$ is a positive constant, $d \alpha_t \to 0$ as $d \to \infty$.
		Therefore, if $t = d^{4+\epsilon}$, $1 - \|S\|_{M_1} - \alpha_t \geq 1/(2d)$ when $d$ is sufficiently large.
		By the bound \eqref{ine:slice-ours}, as $d \to \infty$,
		\[
		1 - \|\hat{S}\|_{M_1} \geq 2^{-1} d^{-5 - \epsilon}.
		\]
		We note that the same asymptotic result can also be obtained through the bound \eqref{ine:slice-previous}.
		
		After the initial release of our work, \citet{power2024weak} further extended the results in this section and \cite{latuszynski2014convergence} regarding hybrid slice samplers.
		
	}

	\section{Discussion} \label{sec:discussion}
	
	To obtain a non-trivial lower bound on $1 - \|\hat{T}\|_{\Pi}$ using Corollary~\ref{cor:random}, it is imperative that $\|Q_{i,y}\|_{\Pi_{i,y}} \leq C$ for all values of $i$ and $y$, where $C$ is some constant that is strictly less than one.
	Relaxing this constraint emerges as a clear avenue for prospective research.
	Techniques based on weak Poincar\'{e} inequalities or the $s$-conductance may be useful; see \citep{andrieu2022comparison} and \citep{ascolani2024scalability}.
	
	The comparative outcomes in Corollaries \ref{cor:random} and \ref{cor:deterministic-2} are useful only if the exact Gibbs sampler in question admits a non-vanishing spectral gap.
	However, even in the absence of a positive spectral gap, the comparison results based on Dirichlet forms, as expounded in Theorems \ref{thm:random} and \ref{thm:deterministic-2}, may lead to valuable insights into the convergence behavior of the hybrid samplers \citep{andrieu2022comparison}.
	Additionally, exploring the comparison in terms of the conductance \citep{lawler1988bounds}, the $s$-conductance \citep{lovasz1993random}, or the approximate spectral gap \citep{atchade2021approximate} could yield further interesting results.
	See, e.g., \cite{ascolani2024scalability}.
	
	{
		An interesting aspect of random-scan Gibbs samplers is the influence of the selection probability	 vector $p = (p_1, \dots, p_n)$.
		Let $p$ and $p' = (p_1', \dots, p'_n)$ be two vectors such that $p_i > 0$, $p_i' > 0$ for each~$i$, and $\sum_{i=1}^n p_i = \sum_{i=1}^n p_i' = 1$.
		Let $T(p)$ and $\hat{T}(p)$ be Mtks defined as \eqref{eq:T} and \eqref{eq:That}, respectively.
		Let $T(p')$ and $\hat{T}(p')$ be defined analogously, with $p_i$ replaced by $p_i'$ for each~$i$.
		By Theorem~1 of \cite{jones2014convergence}, $1 - \|T(p)\|_{\Pi} \geq (\min_i p_i/p_i') (1 - \|T(p')\|_{\Pi})$, and $1 - \|\hat{T}(p)\|_{\Pi} \geq (\min_i p_i/p_i') (1 - \|\hat{T}(p')\|_{\Pi})$.
		Related to this, one reviewer asked the following intriguing question:
		Assuming that $1 - \|T(p)\|_{\Pi} \geq 1 - \|T(p') \|_{\Pi}$, does it follow that $1 - \|\hat{T}(p)\|_{\Pi} \geq 1 - \|\hat{T}(p') \|_{\Pi}$?
		Although we were unable to prove or disprove this conjecture, Corollary~\ref{cor:random} provides a straightforward approach to establish a weaker, but related, result:
		\begin{proposition}
			Suppose that there exists a positive number $b$ such that $1 - \|T(p)\|_{\Pi} \geq b(1 - \|T(p') \|_{\Pi})$.
			Suppose further that there exists a constant $C \in [0,1]$ such that $\|Q_{i,y}\|_{\Pi_{i,y}} \leq C$ for $i = 1,\dots,n$ and $M_{-i}$-a.e. $y \in \X_{-i}$.
			Then
			\[
			1 - \|\hat{T}(p)\|_{\Pi} \geq \frac{b (1-C)}{1+C} (1 - \|\hat{T}(p')\|_{\Pi}).
			\]
			If, furthermore, $Q_{i,y}$ is positive semi-definite for $i = 1,\dots,n$ and $M_{-i}$-a.e. $y \in \X_{-i}$, then
			\[
			1 - \|\hat{T}(p)\|_{\Pi} \geq b(1-C) (1 - \|\hat{T}(p')\|_{\Pi}).
			\]
		\end{proposition}

		It is possible to present the results herein in a more general manner at the price of notational complications only.
		Let $\X$ be a generic Polish space and let $\B$ be its Borel algebra.
		Let $\Y_1, \dots, \Y_n$ be a sequence of Polish spaces, and, for $i = 1,\dots,n$, let $\psi_i: \X \to \Y_i$ be a measurable function.
		Then, if $X \sim \Pi$, one may define the conditional distribution of $X$ given $\psi_i(X) = y$ for $y \in \Y_i$ through disintegration \citep{chang1997conditioning}.
		Denote this conditional distribution by $\Pi^{\star}_{i,y}(\cdot)$.
		Assume that given $y \in \Y_i$, there is an Mtk $Q^{\star}_{i,y}: \X \times \B \to [0,1]$ that is reversible with respect to $\Pi_{i,y}^{\star}$.
		Under measurability conditions, we may define the Mtks
		\[
		T^{\star}(x, \df x') = \sum_{i=1}^n p_i \, \Pi_{i,\psi_i(x)}^{\star}(\df x'), \quad \hat{T}^{\star}(x, \df x') = \sum_{i=1}^n p_i \, Q_{i, \psi_i(x)}^{\star}(x, \df x'),
		\]
		where $p_1, \dots, p_n$ are non-negative constants satisfying $\sum_{i=1}^n p_i = 1$.
		Both Mtks are reversible with respect to $\Pi$.
		Note that if $\X$ has the product form $\X= \prod_{i=1}^n \X_i$, and $\psi_i(x) = x_{-i}$ for $i \in \{1,\dots,n\}$ and $x = (x_1, \dots, x_n) \in \X$, then $T^{\star}$ and $\hat{T}^{\star}$ reduce to $T$ and $\hat{T}$, respectively.
		Following the proof of Theorem \ref{thm:random}, one may establish that theorem in a broader context, with $\Pi_{i,y}$ replaced by $\Pi_{i,y}^{\star}$, $Q_{i,y}$ replaced by $Q_{i,y}^{\star}$, $T$ replaced by $T^{\star}$, and $\hat{T}$ replaced by $\hat{T}^{\star}$.
		
	}

	\begin{appendix}
		
		\section{A Technical Lemma} \label{sec:tec}

		\begin{lemma} \label{lem:Kt}
			Let $K$ be a bounded self-adjoint operator on a Hilbert space $H$ equipped with inner-product $\langle \cdot, \cdot \rangle$ and norm $\|\cdot\|$.
			Then, for any even positive integer~$t$ and $f \in H \setminus \{0\}$,
			\[
			0 \leq \left( \frac{ \langle  f, K f \rangle }{\|f\|^2} \right)^t \leq \frac{ \langle f, K^t f \rangle }{\|f\|^2}.
			\]
			If $K$ is positive semi-definite, then the inequalities also hold for odd positive integers.
		\end{lemma}
		
		\begin{proof}
			Fix $f \in H \setminus \{0\}$.
			Denote by $E_K(\cdot)$ the spectral measure of~$K$ \citep[see, e.g.,][Section 2.7]{arveson2002short}.
			Note that $\langle f, E_K(\cdot) f \rangle$ is a measure on~$\mathbb{R}$, and $\langle  f, E_K(\mathbb{R}) f \rangle = \|f\|^2$.
			Thus,
			\[
			\nu_f(\cdot) = \frac{ \langle f, E_K(\cdot) f \rangle}{ \|f\|^2 }
			\]
			is a probability measure on~$\mathbb{R}$.
			Assume that~$t$ is even or~$K$ is positive semi-definite.
			Then the function $\lambda \mapsto \lambda^t$ is convex on the support of $\nu_f$.
			By the spectral decomposition of~$K$ and Jensen's inequality, 
			\[
			\begin{aligned}
				\left( \frac{ \langle  f, K f \rangle }{\|f\|^2} \right)^t &= \left( \frac{  \int_{\mathbb{R}} \lambda \langle  f, E_K(\df \lambda) f \rangle }{\|f\|^2} \right)^t \\
				&= \left( \int_{\mathbb{R}} \lambda \nu_f (\df \lambda)  \right)^t \\
				&\leq \int_{\mathbb{R}} \lambda^t \nu_f (\df \lambda) \\
				&= \frac{  \langle  f, K^t f \rangle }{\|f\|^2}.
			\end{aligned}
			\]
		\end{proof}

		\section{Proof of Theorem~\ref{thm:deterministic-2}} \label{app:deterministic-2}
		
		Before proceeding to the proof, we first give a useful representation of the operator~$\hat{S}$ using ideas from \cite{latuszynski2014convergence}.

		Define a linear operator $Q: L_0^2(\Pi) \to L_0^2(\Pi)$ as follows:
		for $f \in L_0^2(\Pi)$, 
		\[
		Q f (x) = \int_{\X} Q_{1,x_2}(x_1, \df x_1') \, \delta_{x_2}(\df x_2') \, f(x'), \quad x \in \X.
		\]
		Recall that $Q_{1,z}$ is reversible with respect to $\Pi_{1,z}$ for $z \in \X_2$.
		The following is easy to verify:
		\begin{lemma} \label{lem:Q2positive}
			The operator $Q$ is self-adjoint.
			If $Q_{1,z}$ is positive semi-definite for $M_2$-a.e. $z \in \X_2$, then $Q$ is positive semi-definite.
		\end{lemma}
		
		Define a linear transformation $V: L_0^2(M_1) \to L_0^2(\Pi)$ as follows:
		for $f \in L_0^2(M_1)$, 
		\[
		V f(x) = f(x_1), \quad x \in \X.
		\]
		Define the linear transformation $V^*: L_0^2(\Pi)  \to L_0^2(M_1)$ as follows:
		for $f \in L_0^2(\Pi)$,
		\[
		V^* f (y) = \int_{\X_2} f(y, z) \, \Pi_{2,y}( \df z), \quad y \in \X_1.
		\]
		Then $V^*$ is the adjoint of~$V$ since, for $f \in L_0^2(\Pi)$ and $g \in L_0^2(M_1)$,
		\[
		\begin{aligned}
			\langle f, V g \rangle_{\Pi} &= \int_{\X} f(x) \, g(x_1) \, \Pi(\df x) \\
			&= \int_{\X_1} \int_{\X_2} f(x) \, g(x_1) \, \Pi_{2,x_1}( \df x_2) \, M_1(\df x_1) \\
			&= \int_{\X_1} g(x_1) \int_{\X_2} f(x) \, \Pi_{2,x_1}( \df x_2) \, M_1(\df x_1) \\
			&= \langle V^* f, g \rangle_{M_1}.
		\end{aligned}
		\]

		The following is easy to derive.
		\begin{lemma} \label{lem:St}
			The Mtk on $\X_1 \times \B_1$ given by
			\begin{equation} \label{eq:St}
				\hat{S}_t(y, \df y') = \int_{\X_2} \Pi_{2,y}(\df z) \, Q_{1,z}^t(y, \df y')
			\end{equation}
			corresponds to the self-adjoint operator $V^* Q^t V$ on $L_0^2(M_1)$ for any positive integer~$t$.
			In particular, $\hat{S} = \hat{S}_1 = V^* Q V$.
		\end{lemma}

		We are now ready to prove Theorem~\ref{thm:deterministic-2}, which is restated below.
		
		\vspace{0.2cm}
		{ \sc Theorem~\ref{thm:deterministic-2}.}
		{\it
			Let $\gamma: \X_2 \to [0,1]$ be a measurable function such that $\|Q_{1,z}\|_{\Pi_{1,z}} \leq \gamma(z)$ for $M_2$-a.e. $z \in \X_2$.
			Let~$t$ be some positive integer, and let $\alpha_t \in [0,1]$ be such that
			\begin{equation} \label{ine:alphat}
				\int_{\X_2} \gamma(z)^t \, \Pi_{2,y}( \df z) \leq \alpha_t
			\end{equation}
			for $M_1$-a.e. $y \in \X_1$.
			If either (i) $t$ is even, or (ii) $Q_{1,z}$ is positive semi-definite for $M_2$-a.e. $z \in \X_2$, then, for $f \in L_0^2(M_1) \setminus \{0\}$,
			\begin{equation} \tag{\ref{ine:Shat-Dir-2}}
				0 \leq \frac{\langle f, \hat{S} f \rangle_{M_1}^t}{\|f\|_{M_1}^{2t}} \leq \frac{\langle f, S f \rangle_{M_1}}{\|f\|_{M_1}^2} + \alpha_t.
			\end{equation}
		}
		\begin{proof}
			Let~$t$ be a positive integer for which~\eqref{ine:alphat} holds, and let $f \in L_0^2(M_1)$ be nonzero.
			Let $\hat{S}_t$ be defined as in~\eqref{eq:St}.
			Replacing $Q_{1,z}$ with $Q_{1,z}^t$ in~Lemma \ref{lem:andrieu-0} yields
			\[
			\begin{aligned}
				\mathcal{E}_{\hat{S}_t}( f) &= \int_{\X_2} M_2(\df z) \, \mathcal{E}_{Q_{1,z}^t}( f) \\
				&= \int_{\X_2} M_2(\df z) \, \mathcal{E}_{Q_{1,z}^t}( f - \Pi_{1,z} f) \\
				&\geq \mathcal{E}_{S}( f) - \int_{\X_2} M_2(\df z) \gamma(z)^t \|f - \Pi_{1,z} f\|_{\Pi_{1,z}}^2.
			\end{aligned}
			\]
			The inequality in the final line holds because, by self-adjoint-ness, $\|Q_{1,z}^t\|_{\Pi_{1,z}} = \|Q_{1,z}\|_{\Pi_{1,z}}^t \leq \gamma(z)^t$.
			Now,
			\[
			\begin{aligned}
				\int_{\X_2} M_2(\df z) \gamma(z)^t \|f - \Pi_{1,z} f\|_{\Pi_{1,z}}^2 &\leq \int_{\X_2} M_2(\df z) \gamma(z)^t \|f \|_{\Pi_{1,z}}^2 \\
				&= \int_{\X} \Pi(\df x) \gamma(x_2)^t f(x_1)^2 \\
				&= \int_{\X_1} M_1(\df x_1) f(x_1)^2 \int_{\X_2} \Pi_{2,x_1}(\df x_2) \gamma(x_2)^t \\
				&\leq \alpha_t \|f\|_{M_1}^2.
			\end{aligned}
			\]
			Combining terms, we have
			\[
			\mathcal{E}_{\hat{S}_t}( f) \geq \mathcal{E}_{S}( f) - \alpha_t \|f\|_{M_1}^2, 
			\]
			i.e.,
			\[
			\langle f, \hat{S}_t f \rangle_{M_1} \leq \langle f, S f \rangle_{M_1} + \alpha_t \|f\|_{M_1}^2.
			\]

			Assume that~$t$ is even, or that $Q_{1,z}$ is positive semi-definite for $M_2$-a.e. $z \in \X_2$.
			Then, by Lemma~\ref{lem:Q2positive}, $t$ is even, or $Q$ is positive semi-definite.
			Note also that, for $f, g \in L_0^2(M_1)$, $\langle f, g \rangle_{M_1} = \langle Vf, Vg \rangle_{\Pi}$.
			By Lemmas~\ref{lem:Kt} and~\ref{lem:St},
			\[
			\begin{aligned}
				0 \leq \left( \frac{\langle f, \hat{S} f \rangle_{M_1}}{\|f\|_{M_1}^2} \right)^t =  \left( \frac{\langle  V f, Q V f \rangle_{\Pi}}{\|V f\|_{\Pi}^2} \right)^t   \leq \frac{\langle V f, Q^t V f \rangle_{\Pi}}{\|V f\|_{\Pi}}   = \frac{\langle  f, \hat{S}_t f \rangle_{M_1}}{\|f\|_{M_1}^2}.
			\end{aligned}
			\]
			
			Combining the two most recent displays yields~\eqref{ine:Shat-Dir-2}.
		\end{proof}

		\section{Proof of Proposition \ref{pro:metro}} \label{app:metro}
		\begin{proof}
			Fix $z \in \{0,1\}^d$ and $q \in (0,1)$.
			For $\beta \in \X_1 = \mathbb{R}^d$, let $v_z(\beta) = (\beta_j/\sqrt{\tau(z_j)})_{j=1}^d$.
			Define the Mtk $\tilde{Q}_{1,(z,q)}: \X_1 \times \B_1 \to [0,1]$, where $\B_1$ is the Borel algebra of $\mathbb{R}^d$, as follows: for $\alpha \in \X_1$ and $A \in \B_1$, let
			\[
			\tilde{Q}_{1,(z,q)}(\alpha, A) = Q_{1,(z,q)}(v_z^{-1}(\alpha), v_z^{-1}(A)).
			\]
			Then $\tilde{Q}_{1,(z,q)}$ defines a self-adjoint operator on $L_0^2(\Pi_{1,(z,q)} \circ v_z^{-1})$.
			One can verify that, for $g \in L_0^2(\Pi_{1,(z,q)}) \setminus \{0\}$, it holds that $g \circ v_z^{-1} \in L_0^2(\Pi_{1,(z,q)} \circ v_z^{-1}) \setminus \{0\}$, and
			\begin{equation} \label{eq:metro-1}
				\frac{ \mathcal{E}_{Q_{1,(z,q)}}(g) }{\|g\|_{\Pi_{1,(z,q)} }^2 } = \frac{ \mathcal{E}_{\tilde{Q}_{1,(z,q)}} (g \circ v_z^{-1})  }{\|g \circ v_z^{-1} \|_{\Pi_{1,(z,q)} \circ v_z^{-1} }^2 }.
			\end{equation}

			The Mtk $\tilde{Q}_{1,(z,q)}$ corresponds to the random-walk Metropolis algorithm targeting $\Pi_{1,(z,q)} \circ v_z^{-1}$ whose proposal distribution, given the current state $\alpha$, is $\mbox{N}_d(\alpha, \sigma_d^2 I_d)$.
			It can be checked that the log density of $\Pi_{1,(z,q)} \circ v_z^{-1}$ is 
			\[
			\alpha \mapsto \ell(\sqrt{\tau(z_1)}\alpha_1, \dots, \sqrt{\tau(z_d)}\alpha_d) - \sum_{j=1}^d \frac{\alpha_j^2}{2} + \mbox{constant}.
			\]
			Let $\tilde{H}(\alpha)$ be the Hessian of this function.
			Then it holds that 
			\[
			\tilde{H}(\alpha) = M_z^{1/2} H(v_z^{-1}(\alpha)) M_z^{1/2} - I_d,
			\]
			where $M_z$ is the $d \times d$ diagonal matrix whose $j$th diagonal element is $\tau(z_j)$.
			Hence, under Condition \ref{H1}, 
			\[
			I_d \preccurlyeq -\tilde{H}(\alpha) \preccurlyeq (b_d L_d + 1) I_d.
			\]
			Then, by Theorem 1 of \cite{andrieu2022explicit}, for $h \in L_0^2(\Pi_{1,(z,q)} \circ v_z^{-1})$,
			\[
			\frac{ \mathcal{E}_{\tilde{Q}_{1,(z,q)}}(h) }{\|h\|_{\Pi_{1,(z,q)} \circ v_z^{-1}}^2 } \geq \frac{C_0}{d(b_d L_d + 1)},
			\]
			where $C_0$ is a universal constant.
			Moreover, by Lemma 3.1 of \cite{baxendale2005renewal}, $\tilde{Q}_{1,(z,q)}$ is positive semi-definite, so, 
			for $h \in L_0^2(\Pi_{1,(z,q)} \circ v_z^{-1})$,
			\[
			\frac{ \mathcal{E}_{\tilde{Q}_{1,(z,q)}}(h) }{\|h\|_{\Pi_{1,(z,q)} \circ v_z^{-1}}^2 } \leq 1.
			\]
			The desired result then follows from \eqref{eq:metro-1}.
		\end{proof}
		
		\section{Proof of Proposition \ref{pro:latu}} \label{app:latu}
		
		\begin{proof}
			Since $0 \leq \alpha_t \leq \beta_t$, it suffices to show that $\beta_t \to 0$ as $t \to \infty$.
			
			Suppose the opposite.
			Then there exist a positive number $\beta_*$ and a strictly increasing sequence of positive integers $(t_j)_{j=1}^{\infty}$ such that, for $j = 1,2, \dots$,
			\[
			\sup_{m \in (0, \|\tilde{m}\|_{\infty})} \frac{1}{m} \int_0^m \gamma(z)^{2 t_j} \, \df z > \beta_*.
			\]
			Then, for each $j$, one can find $m_j \in (0, \|\tilde{m}\|_{\infty})$ such that
			\begin{equation} \label{ine:betastar}
				\frac{1}{m_j} \int_0^{m_j} \gamma(z)^{2 t_j} \, \df z > \beta_*.
			\end{equation}
			We now use an argument from \cite{latuszynski2014convergence}.
			By assumption, $m_j > \epsilon$ when $j$ is sufficiently large, for otherwise
			\[
			\liminf_{j \to \infty} \frac{1}{m_j} \int_0^{m_j} \gamma(z)^{2 t_j} \, \df z \leq \liminf_{j \to \infty} \gamma_0^{2 t_j} = 0,
			\]
			violating \eqref{ine:betastar}.
			Then, for large enough $j$, the function $z \mapsto (1/m_j) \gamma(z)^{2 t_j} \ind(z \in (0,m_j))$ is upper bounded by $z \mapsto 1/\epsilon$, which is integrable on $(0,\|\tilde{m}\|_{\infty})$.
			Therefore, by the dominated convergence theorem,
			\[
			\begin{aligned}
				\limsup_{j \to \infty} \frac{1}{m_j} \int_0^{m_j} \gamma(z)^{2 t_j} \, \df z &= \limsup_{j \to \infty}  \int_0^{\|\tilde{m}\|_{\infty}} \frac{\gamma(z)^{2 t_j} \ind(z \in (0, m_j))}{m_j} \, \df z \\
				&\leq \limsup_{j \to \infty}  \int_0^{\|\tilde{m}\|_{\infty}} \frac{\gamma(z)^{2 t_j} }{\epsilon} \, \df z \\
				&= \int_0^{\|\tilde{m}\|_{\infty}} \lim_{j \to \infty} \frac{\gamma(z)^{2 t_j} }{\epsilon} \, \df z = 0.
			\end{aligned}
			\]
			But this again violates \eqref{ine:betastar}.
			Thus, it must hold that $\lim_{t \to \infty} \beta_t = 0$.
		\end{proof}

	\end{appendix}
	
	\begin{acks}[Acknowledgments]
		We thank an Editor, an Associate Editor and two Referees for helpful comments.
		One Referee provided valuable suggestions that helped streamline the proof of Theorem \ref{thm:random}.
		We gratefully acknowledge the 2023 IRSA Conference at the University of Minnesota for catalyzing the collaborative efforts among us.
	\end{acks}
	\begin{funding}
		The first author was partially supported by the National Science Foundation through grant DMS-2112887.
		The third author was partially supported by the National Science Foundation through grant DMS-2210849 and an Adobe Data Science Research Award.
	\end{funding}
	
	\begin{supplement} 
		\stitle{Supplement to "Spectral gap bounds for reversible hybrid Gibbs chains"}
		\sdescription{The supplementary file contains two sections.
			Section I provides an analysis of a Metropolis-within-proximal sampler based on Proposition \ref{pro:andrieu-0}, taken from \cite{andrieu2018uniform}.
			When the target distribution is well-conditioned, it is shown that the spectral gap of the hybrid proximal sampler is bounded from below by a multiple of $(m/L)(1/d)$, where $d$ is the dimension, and $L/m$ is the conditional number of the negative log of the target density.
			Section II provides an alternative quantitative relationship between $\|T\|_{\Pi}$ and $\|\hat{T}\|_{\Pi}$ based on \cite{roberts1997geometric}.
			It is then compared to the quantitative relation given in Corollary \ref{cor:random}.}
	\end{supplement}
	

	\bibliographystyle{imsart-nameyear} 
	\bibliography{qinbib}       
	
	
	
	\newpage
	
	\begin{center}
		{\bf \Large Supplement to "Spectral gap bounds for reversible hybrid Gibbs chains"}
	\end{center}
	
	\vspace{1cm}
	
	\begin{center}
		\begin{minipage}{0.86\textwidth}
			{\small 
			This supplementary file contains two sections.
			
			Section \ref{sec:pro} provides an analysis of a Metropolis-within-proximal sampler based on Proposition \ref{pro:andrieu-0} of the main text, taken from \cite{andrieu2018uniform}.
			When the target distribution is well-conditioned, it is shown that the spectral gap of the hybrid proximal sampler is bounded from below by a multiple of $(m/L)(1/d)$, where $d$ is the dimension, and $L/m$ is the conditional number of the negative log of the target density.
			
			Section \ref{app:roberts} provides an alternative quantitative relationship between $\|T\|_{\Pi}$ and $\|\hat{T}\|_{\Pi}$, where $T$ and $\hat{T}$ are the exact and hybrid random-scan Gibbs kernels defined  in \eqref{eq:T} and \eqref{eq:That} of the main text.
			The relationship is based on arguments from \cite{roberts1997geometric}.
			It is then compared to the quantitative relation given in Corollary \ref{cor:random} of the main text.}
		\end{minipage}

	\end{center}
	
	\vspace{1cm}
	
	\section{A proximal sampler} \label{sec:pro}
	
	This section contains an application of Proposition \ref{pro:andrieu-0} from the main text, taken from \cite{andrieu2018uniform}.
	This application, which involves hybrid proximal samplers, may be interesting in its own right.
	
	Suppose that $\X_1 = \mathbb{R}^d$, where $d$ is a positive integer.
	Suppose that $M_1$, the distribution of interest, admits a density function that is proportional to $y \mapsto e^{-\xi(y)}$, where $\xi: \X_1 \to \mathbb{R}$ is twice differentiable.
	Assume that there exist constants $m \in \mathbb{R}$ and $L \in [0,\infty)$ such that, for $y, u \in \mathbb{R}^d$,
	\begin{equation} \nonumber
		\frac{m}{2} \|u\|^2 \leq \xi(y + u) - \xi(y) - \langle \nabla \xi(y), u \rangle \leq \frac{L}{2} \|u\|^2.
	\end{equation}

	Now, let $\X_2 = \mathbb{R}^d$.
	For $y \in \mathbb{R}^d$, let $\Pi_{2,y}$ be a Gaussian distribution whose density is
	\[
	z \mapsto (2 \pi \eta)^{-d/2} \exp \left( - \frac{\|z - y\|^2}{2 \eta} \right),
	\]
	where $\eta$ is a positive constant that can be tuned.
	Then the joint distribution~$\Pi$ has a density proportional to
	\[
	(y,z) \mapsto \exp \left[ - \xi(y) - \frac{\|z - y\|^2}{2 \eta}  \right].
	\]
	The measure~$M_2$ has a density that is proportional to the above formula with $y$ integrated out.
	For $M_2$-a.e. $z \in \mathbb{R}^d$, $\Pi_{1,z}$ has density function
	\[
	y \mapsto \left\{ \int_{\mathbb{R}^d} \exp \left[ - \xi(y') - \frac{\|z - y'\|^2}{2 \eta}  \right] \, \df y' \right\}^{-1} \exp \left[ - \xi(y) - \frac{\|z - y\|^2}{2 \eta}  \right]  .
	\]
	
	In this case, the data augmentation chain associated with the Markov transition kernel (Mtk)~$S$, where $S(y, \df y') = \int_{\mathbb{R}^d} \Pi_{2,y}(\df z) \, \Pi_{1,z}(\df y')$
	is called a proximal sampler \citep{lee2021structured}.
	It can be particularly useful when $M_1$ is ill-conditioned, although later on in our analysis of its hybrid version, we focus on the case where $M_1$ is well-conditioned for tractability.
	This sampler was studied by \citet{lee2021structuredarxiv} and \citet{chen2022improved}.
	In particular, the following holds.
	\begin{proposition}\citep[][Theorem~4]{chen2022improved} \label{pro:lee}
		Suppose that $m > 0$.
		Then $\|S\|_{M_1} \leq 1/(1+m \eta)$.
	\end{proposition}

	In practice, the conditional distribution $\Pi_{1,z}(\cdot)$ may be difficult to sample from exactly.
	Let us now consider the hybrid version of the data augmentation chain whose Mtk has the form
	\[
	\hat{S}(y, \df y') = \int_{\mathbb{R}^d} \Pi_{2,y}(\df z) \, Q_{1,z}(y, \df y') = (2 \pi \eta)^{-d/2} \int_{\mathbb{R}^d}  \exp \left( - \frac{\|z - y\|^2}{2 \eta} \right) Q_{1,z}(y, \df y').
	\]
	The approximating kernel $Q_{1,z}$ is assumed to be reversible with respect to $\Pi_{1,z}$, and may correspond to a number of algorithms.
	For illustration, take $Q_{1,z}$ to be the Mtk of a random-walk Metropolis algorithm.
	To be specific, the new state is proposed via a $d$-dimensional normal distribution whose mean is the current state, and whose variance is $2^{-1} d^{-1} (L+1/\eta)^{-1} I_d$.
	The following lemma follows directly from Theorem~1 of \cite{andrieu2022explicit} and Lemma 3.1 of \cite{baxendale2005renewal}.
	
	\begin{lemma} \label{lem:andrieu}
		Suppose that $m + 1/\eta > 0$, and that $Q_{1,z}$ corresponds to the Metropolis algorithm above.
		Then, for $z \in \mathbb{R}^d$ and $g \in L_0^2(\Pi_{1,z}) \setminus \{0\}$,
		\[
		\frac{C_0 (m \eta + 1)}{d(L \eta + 1)} \leq \frac{\mathcal{E}_{Q_{1,z}}(g) }{ \|g\|_{\Pi_{1,z}}^2 } \leq 1.
		\]
		where $C_0$ is a positive universal constant.
	\end{lemma}

	Proposition \ref{pro:andrieu-0} and Corollary \ref{cor:deterministic-2} (both from the main text) provide bounds on $1 - \|\hat{S}\|_{M_1}$ in terms of $1 - \|S\|_{M_1}$ and $\|Q_{1,z}\|_{\Pi_{1,z}}$.
	In particular, combining the bound in Proposition \ref{pro:andrieu-0} of the main text with Lemma~\ref{lem:andrieu} yields the following.
	\begin{proposition}\label{pro:proximal-lower}
		Suppose that $m + 1/\eta > 0$, and that $Q_{1,z}$ corresponds to the Metropolis algorithm above.
		Then
		\[
		\frac{C_0 (m \eta + 1)}{d (L\eta + 1)} (1 - \|S\|_{M_1}) \leq 1 - \|\hat{S}\|_{M_1} \leq 1 - \|S\|_{M_1},
		\]
		where $C_0$ is a universal constant.
		In particular, if $\eta = 1/L$,
		\[
		1 - \|\hat{S}\|_{M_1} \geq \frac{C_0}{2 d} (m/L + 1) (1 - \|S\|_{M_1}).
		\]
	\end{proposition}
	
	Combining Proposition~\ref{pro:lee} with Proposition~\ref{pro:proximal-lower}, we immediately have:
	
	\begin{proposition}\label{pro:proximal-lower-optparam}
		Suppose that $m > 0$, $\eta = 1/L$, and that $Q_{1,z}$ corresponds to the Metropolis algorithm above.
		Then
		\[
		1 - \|\hat{S}\|_{M_1} \geq \frac{C_0}{2d} \frac{m}{L}.
		\]
	\end{proposition}
	
	Proposition \ref{pro:proximal-lower-optparam} demonstrates that the hybrid data augmentation sampler achieves a spectral gap that is no smaller than a constant multiple of $(\kappa d)^{-1}$, where $\kappa = L/m$ is the conditional number of $\xi$. 
	We can readily extend these bounds to analyze the mixing time of the underlying Markov chain using conventional techniques \citep[see, e.g.,][Theorem 2.1]{roberts1997geometric}. Recall that an initial distribution $\pi_0$ is said to be  $\beta$-warm with respect to the stationary distribution~$M_1$ if $\pi_0(A)/M_1(A) \leq \beta$ for every measurable $A$ \citep{dwivedi2019log}. Assuming the hybrid sampler begins with a $\beta$-warm initialization, we can deduce from Theorem 2.1 of \cite{roberts1997geometric} that its $\epsilon$-mixing time, measured in both the total variation and $L^2$ distance, is on the order of $O(\kappa d \log\left(\beta/\epsilon\right))$ iterations.
	This exhibits the same order of magnitude as the mixing time bound for the Metropolis adjusted Langevin Algorithm (MALA) in \cite{dwivedi2019log}.
	
	A frequently adopted initialization choice is $\pi_0 = N_d(y^\star, L^{-1} 
	I_d)$, where $y^\star$ is the unique maximizer of $y \mapsto e^{-\xi(y)}$ (often obtainable through optimization methods with minimal computational overhead). 
	This particular $\pi_0$ corresponds to $\beta = O(\kappa^{d/2})$, resulting in an $\epsilon$-mixing time of $O\left(\kappa d (d\log(\kappa) + \log(\epsilon^{-1}))\right)$.

	The above mixing time is of the same order as a method proposed by \citet{chen2022improved}, which implements the exact proximal sampler with $\eta \approx 1/(Ld)$ through a meticulous rejection sampling technique to draw from $\Pi_{1,z}(\cdot)$.
	(See page 7 of that work.)
	As noted in \cite{chen2022improved}, this cost is in some sense state-of-the-art. 
	It is worth mentioning that the hybrid sampler herein may be easier to implement, as the method in \cite{chen2022improved} requires accurately solving a $d$-dimensional convex optimization problem in every iteration.
	
	{
		\section{A bound from Roberts and Rosenthal (1997)} \label{app:roberts}
		
		As mentioned in Section \ref{sec:intro} of the main manuscript, \cite{roberts1997geometric} is a pioneering work on the convergence properties of hybrid Gibbs samplers.
		\citet{roberts1997geometric} did not provide explicit bounds relating $\|T\|_{\Pi}$ to $\|\hat{T}\|_{\Pi}$,  where $T$ and $\hat{T}$ are the exact and hybrid random-scan Gibbs kernels defined  in \eqref{eq:T} and \eqref{eq:That} of the main text. 
		However, it is possible to use their argument to obtain a quantitative relationship.
		In this section, we derive such a relation, and compare it to Corollary \ref{cor:random} of the main text.
		
		We first review the relevant result in \cite{roberts1997geometric}.
		We adopt the notations from Section \ref{ssec:samplers} of the main text.
		For $i = 1, \dots, n$, let $\Pi_i: L_0^2(\Pi) \to L_0^2(\Pi)$ be such that $\Pi_i f(x) = \int_{\X} \Pi_{i,x_{-i}}(\df x_i') \, \delta_{x_{-i}} (\df x_{-i}') \, f(x')$, and let $Q_i: L_0^2(\Pi) \to L_0^2(\Pi)$ be such that $Q_i f(x) = \int_{\X} Q_{i,x_{-i}}(x_i, \df x_i') \, \delta_{x_{-i}} (\df x_{-i}') \, f(x')$.
		Then, $T$ and $\hat{T}$, as defined in \eqref{eq:T} and \eqref{eq:That} of the main text, can be written as $T = \sum_{i=1}^n p_i \Pi_i$ and $\hat{T} = \sum_{i=1}^n p_i Q_i$, respectively.
		
		\begin{proposition}\citep[][Theorem 3.1]{roberts1997geometric} \label{pro:roberts}
			Suppose that $p_1 = \cdots = p_n = 1/n$.
			Suppose further that, for $i \in \{1, \dots, n\}$, $\|Q_i^t - \Pi_i\|_{\Pi} \to 0$ as $t \to \infty$.
			Then $\|\hat{T}\|_{\Pi} < 1$ whenever $\|T\|_{\Pi} < 1$.
		\end{proposition}
		
		\begin{remark}
			The original theorem is stated for operator norms that satisfy a contracting property, which is satisfied by $\|\cdot\|_{\Pi}$.
			\citet{roberts1997geometric} also did not require that $Q_{i,y}$ must reversible with respect to $\Pi_{i,y}$ for each value of $(i,y)$, which is assumed herein.
		\end{remark}
		
		The original proof of Proposition \ref{pro:roberts} roughly goes as follows.
		We will keep track of intermediate parameters in the original proof to form a quantitative bound.
		Fix a positive integer~$t$.
		Then one may write $\hat{T}^t$ as 
		\[
		n^{-t} \left( \sum_{i=1}^n Q_i \right)^t  = n^{-(t-1)} n^{-1} \sum_{i=1}^n Q_i^t + \left( 1 - n^{-(t-1)} \right) R_{t}, 
		\]
		where $R_{t}$ is some Mtk satisfying $\Pi R_{t} = \Pi$.
		By the triangle inequality and the fact that $\|R_{t}\|_{\Pi} \leq 1$,
		\[
		\|\hat{T}\|_{\Pi} \leq n^{-(t-1)} \left\| n^{-1} \sum_{i=1}^n Q_i^t \right\|_{\Pi} + 1 - n^{-(t-1)} \leq n^{-(t-1)} \|T\|_{\Pi} + n^{-t} \sum_{i=1}^n C_{i,t} + 1 - n^{-(t-1)},
		\]
		where $C_{i,t} = \|Q_i^t - \Pi_i\|_{\Pi}$.
		We may then obtain the following quantitative relation,  which, while not explicitly stated in \cite{roberts1997geometric}, is implied in the derivations therein:
		\begin{equation} \label{ine:roberts-bound-1}
			1 - \|\hat{T}\|_{\Pi} \geq n^{-(t-1)} (1 - \|T\|_{\Pi}) - n^{-t} \sum_{i=1}^n C_{i,t}.
		\end{equation}
		Under the assumptions in Proposition \ref{pro:roberts}, one may find $t$ sufficiently large so that 
		\[
		n^{-t} \sum_{i=1}^n C_{i,t} < n^{-(t-1)} (1 - \|T\|_{\Pi}).
		\]
		At this point, Proposition \ref{pro:roberts} is immediate.
		
		The quantity $C_{i,t}$ is slightly mysterious.
		Using the notation $f_{i,x_{-i}}(x_i) = f(x)$ in the proof of Theorem \ref{thm:random} from the main text, we may write it as follows:
		\[
		\begin{aligned}
			C_{i,t} &= \sqrt{\sup_{f \in L_0^2(\Pi) \setminus \{0\}} \frac{\int_{\X} [Q_i^t f(x) - \Pi_i f(x)]^2 \, \Pi(\df x) }{\|f\|_{\Pi}^2} } \\
			&= \sqrt{\sup_{f \in L_0^2(\Pi) \setminus \{0\}} \frac{\int_{\X_{-i}} \int_{\X_i} [Q_{i,x_{-i}}^t (f_{i,x_{-i}} - \Pi_{i,x_{-i}} f_{i,x_{-i}}) (x_i)]^2 \, \Pi_{i,x_{-i}}(\df x_i) \, M_{-i}(\df x_{-i}) }{\|f\|_{\Pi}^2} }.
		\end{aligned}
		\]
		Assume, as in Corollary \ref{cor:random} of the main text, that $\|Q_{i,y}\|_{\Pi_{i,y}} \leq C$ for some constant $C \in [0,1]$.
		Then
		\[
		\begin{aligned}
			C_{i,t} &\leq C^t \sqrt{\sup_{f \in L_0^2(\Pi) \setminus \{0\}} \frac{\int_{\X_{-i}} \int_{\X_i} [f_{i,x_{-i}}(x_i) - \Pi_{i,x_{-i}} f_{i,x_{-i}}]^2 \, \Pi_{i,x_{-i}}(\df x_i) \, M_{-i}(\df x_{-i}) }{\|f\|_{\Pi}^2} } \\
			&= C^t \sqrt{\sup_{f \in L_0^2(\Pi) \setminus \{0\}} \frac{\int_{\X} [f(x) - \Pi_i f(x)]^2 \, \Pi(\df x) }{\|f\|_{\Pi}^2} }.
		\end{aligned}
		\]
		The supremum is attained if $f$ is orthogonal to the range of $\Pi_i$, which is an orthogonal projection \citep{greenwood1998information}, in which case the last line becomes $C^t$.
		Hence, $C_{i,t} \leq C^t$, and, by \eqref{ine:roberts-bound-1},
		\begin{equation} \label{ine:roberts-bound-2}
			1 - \|\hat{T}\|_{\Pi} \geq n^{-(t-1)} (1 - \|T\|_{\Pi} - C^t).
		\end{equation}
		
		Routine calculations show that, when $n \geq 2$ and $t \geq 1$,
		\[
		n^{-(t-1)} (1 - \|T\|_{\Pi} - C^t) \leq n^{-(t-1)} (1-C^t)(1 - \|T\|_{\Pi}) \leq (1-C)(1-\|T\|_{\Pi}).
		\]
		Hence, \eqref{ine:roberts-bound-2} is looser than the lower bound in \eqref{ine:random-bound-1} from Corollary \ref{cor:random} of the main manuscript.
		For a rough quantitative comparison, fix $T$ and assume that $C = 1 - \varepsilon$ where $\varepsilon = o(1)$.
		Take $t = \lceil - \varepsilon^{-1} \log [(1-\|T\|_{\Pi})/2] + O(1)$.
		Then
		\[
		n^{-(t-1)} (1 - \|T\|_{\Pi} - C^t) = \Theta \left[ \left( \frac{1 - \|T\|_{\Pi}}{2} \right)^{(\log n)/\varepsilon} \right],
		\]
		where $\Theta$ means the left and righ hand sides are of the same order.
		Then the lower bound in \eqref{ine:roberts-bound-2} is exponential in $-1/\varepsilon$, while the lower bound in \eqref{ine:random-bound-1} of the main text is proportional to $\varepsilon$.
		In the example in Section \ref{sssec:generic} of the main text, $\varepsilon$ is inversely proportional to a polynomial of some dimension parameters.
		In that example, if $k = n$, $m_{i,y}/L_{i,y}$ is lower bounded by a positive constant $1/\kappa$, and $d_i = d$ for $i = 1,\dots,n$, then one may take $\varepsilon = C_0/(\kappa d)$.
		In this case, \eqref{ine:random-bound-1} gives the bound $1 - \|\hat{T}\|_{\Pi} \geq C_0 / (\kappa d) (1 - \|T\|_{\Pi}) $, while \eqref{ine:roberts-bound-2} gives $1 - \|\hat{T}\|_{\Pi} \geq c [(1-\|T\|_{\Pi})/2]^{(\log n)\kappa d / C_0} $ when $d$ is large, where $c$ is a positive constant.
		Evidently, the former is much tighter.
		
		It may be possible to improve the bound \eqref{ine:roberts-bound-2}, e.g., by bounding $C_{i,t}$ some other way. 
		We will not pursue this here.
		
	}

\end{document}